\documentclass[11pt]{amsart}
  \newcommand{\LE}{\mathcal{LE}}
   \renewcommand{\H}{\mathcal{H}}
  \newcommand{\B}{\mathcal{B}}

\newcommand{\G}{\mathcal{G}}
\newcommand{\cZ}{\mathcal{Z}}
\newcommand{\X}{\mathcal{X}}
\newcommand{\Y}{\mathcal{Y}}
\newcommand{\gG}{\Gamma}
\newcommand{\C}{\mathbb{C}}
\newcommand{\T}{\mathbb{T}}

\newcommand{\R}{\mathbb{R}}
\newcommand{\E}{\mathbb{E}}
\newcommand{\N}{\mathbb{N}}

\newcommand{\Z}{\mathbb{Z}}

\newcommand{\norm}[1]{\left\Vert #1\right\Vert}
\newcommand{\nnorm}[1]{\lvert\!|\!| #1|\!|\!\rvert}

\theoremstyle{plain}
\newtheorem{theorem}{Theorem}[section]
\newtheorem{lemma}[theorem]{Lemma}
\newtheorem{proposition}[theorem]{Proposition}

\newtheorem*{conjecture2'}{Problem 4$'$}

\newtheorem*{theoremA'}{Theorem A'}
\newtheorem*{theoremB'}{Theorem B'}
\newtheorem*{theoremC'}{Theorem C'}

\newtheorem*{theorem*}{Theorem}
\newtheorem*{Correspondence1}{Furstenberg Correspondence Principle}

\newtheorem{conjecture}{Problem}

\theoremstyle{definition}

\newtheorem{example}{Example}

\theoremstyle{remark}

\newtheorem*{remark}{Remark}
\newtheorem*{remarks}{Remarks}

\begin{document}
\title{Multiple recurrence and convergence for Hardy  sequences of polynomial growth}

\author{Nikos Frantzikinakis}
\address[Nikos  Frantzikinakis]{Department of Mathematics\\
  University of Crete\\Knossos Avenue\\
    Heraklion\\ 71409 \\ GREECE } \email{frantzikinakis@gmail.com}

\begin{abstract}
We study the limiting behavior of multiple ergodic averages involving  sequences of integers
   that satisfy some regularity conditions and have polynomial growth.
   We show that for  ``typical''
   choices of
   Hardy field functions $a(t)$     with  polynomial growth,
 the averages  $\frac{1}{N}\sum_{n=1}^N f_1(T^{[a(n)]}x)\cdot\ldots\cdot f_\ell(T^{\ell [a(n)]}x)$ converge in the mean and  we determine  their limit.  For example, this is the case if $a(t)=t^{3/2}, t\log{t},$ or $t^2+(\log{t})^2$.
 Furthermore, if $\{a_1(t),\ldots,a_\ell(t)\}$ is a ``typical'' family of logarithmico-exponential functions of polynomial growth, then  for every ergodic system, the
averages $\frac{1}{N}\sum_{n=1}^N f_1(T^{[a_1(n)]}x)\cdot\ldots\cdot f_\ell(T^{[a_\ell(n)]}x)$
converge in the mean to the product of the integrals of the corresponding functions.
For example, this is the case if the functions $a_i(t)$ are given by different positive fractional powers of $t$.
We deduce several results in combinatorics. We show
that if $a(t)$ is a non-polynomial Hardy field function
with polynomial growth,   then every
set of integers with positive upper density contains arithmetic progressions of the form $\{m,m+[a(n)],\ldots,m+\ell[a(n)]\}$.
Under suitable assumptions we get a related result
concerning patterns of the form
$
\{m, m+[a_1(n)],\ldots, m+[a_\ell(n)]\}.
$
\end{abstract}

\thanks{The  author was partially supported by NSF grant
 DMS-0701027 and Marie Curie International Reintegration Grant 248008.}



\maketitle


\setcounter{tocdepth}{1}
\tableofcontents
\section{Introduction}
In recent years there has been a lot of activity in studying the limiting behavior in $L^2(\mu)$ (as $N\to\infty$) of
multiple ergodic averages of the form
\begin{equation}\label{E:basic1}
\frac{1}{N}\sum_{n=1}^N f_1(T^{a_1(n)}x)\cdot\ldots\cdot f_\ell(T^{a_\ell(n)}x)
\end{equation}
for various choices of sequences of integers $a_1(n), \ldots, a_\ell(n)$, where $T$ is an invertible measure preserving transformation acting on a probability space $(X,\X,\mu)$, and $f_1,\ldots,f_\ell$ are bounded measurable functions. This study was initiated
in \cite{Fu77}, where Furstenberg  studied the averages \eqref{E:basic1} when $a_1(n)=n, a_2(n)=2n,\ldots, a_\ell(n)=\ell n$,
in a depth   that was sufficient to give a new proof of Szemer\'edi's theorem on arithmetic progressions (\cite{Sz75}). Later on, Bergelson and Leibman in \cite{BL96} extended Furstenberg's method to cover the case
where the sequences  $a_1(n), \ldots, a_\ell(n)$ are integer polynomials with zero constant term,
 and established a polynomial extension of Szemer\'edi's theorem.

 For $\ell=1$ the limiting behavior in $L^2(\mu)$ of the averages \eqref{E:basic1} can be reduced (using the spectral theorem for unitary operators) to the study of certain exponential sums, and therefore is in a sense well understood.
 For $\ell\geq 2$, even in the simplest cases, convergence
of the averages \eqref{E:basic1} and identification of the limit turned out to be a very resistent problem. Nevertheless, we now have several different proofs of convergence in the case where   the sequences are linear (\cite{HK05a}, \cite{Zi07}, \cite{Ta08}, \cite{To09},\cite{Au09}, \cite{Ho09}), and the case where all the sequences are integer polynomials  was treated in \cite{HK05b} and \cite{Lei05c}. Furthermore, a rather explicit formula for the limit of the averages \eqref{E:basic1} can  be given in the linear case (combining results from \cite{HK05a} and  \cite{Zi05}), and for some special collections of integer polynomial sequences   (\cite{FrK06}, \cite{Fr08}, \cite{Lei09}).

The purpose of this article is to carry out a detailed  study of the limiting behavior of the averages \eqref{E:basic1}   for a large class of sequences of integers $a_1(n),\ldots, a_\ell(n)$ that
 have  polynomial growth (meaning $a_i(t)/t^k\to 0$ for some $k\in\N$) but are not necessarily defined by integer polynomials.\footnote{The case where the transformation $T$ is a nilrotation was treated in the companion paper \cite{Fr09} and is an essential component of the present paper.} For example, we shall show
that for every positive $c\in \R\setminus \Z$, and measure preserving transformation $T$, the averages
\begin{equation}\label{E:basic2}
\frac{1}{N}\sum_{n=1}^N f_1(T^{[n^c]}x)\cdot f_2(T^{2[n^c]}x)\cdot \ldots\cdot f_\ell(T^{\ell[n^c]}x)
\end{equation}
converge in $L^2(\mu)$,  and  their limit is equal to the limit of the ``Furstenberg averages''
\begin{equation}\label{E:Furst}
\frac{1}{N}\sum_{n=1}^N f_1(T^{n}x)\cdot f_2(T^{2n}x)\cdot\ldots\cdot f_\ell(T^{\ell n}x).
\end{equation}
More generally, the role of the sequence $[n^c]$ in \eqref{E:basic2} can play any sequence $[a(n)]$
where $a(t)$ is a function that belongs to some Hardy field, has polynomial growth,
 and stays logarithmically away from constant multiples of integer polynomials (see Theorem~\ref{T:ConvSingleFormula}).
 For instance, any of the following sequences  will work (below, $k$ is an arbitrary positive integer)
 \begin{equation}\label{E:ExamplesConv}
  [n\log n],\ [n^3/\log{n}],
   \  [n^2+n\log n],  \ [n^2+\sqrt{3}\ \! n],\ [n^2+(\log{n})^2],\
[(\log(n!))^k], \ [(\text{Li}(n))^k].
\end{equation}
We also give explicit
necessary and sufficient conditions for mean convergence of the averages \eqref{E:basic2} when the
sequence $[n^c]$ is replaced with the sequence $[a(n)]$
where $a(t)$ is any function that belongs to some Hardy field and  has polynomial growth.

With the help of   the previous convergence results we derive a
 refinement of Szemer\'edi's theorem on arithmetic progressions.
 We show that
if $a(t)$ is a function that belongs to some Hardy field, has polynomial growth, and is
not equal to a constant multiple of an integer polynomial (modulo a function that converges to a constant), then  for every $\ell\in \N$,
every set of integers with positive upper density\footnote{A set of integers $\Lambda$ has \emph{positive upper density} if $\bar{d}(\Lambda)=\limsup_{N\to\infty}|\Lambda\cap
\{-N,\ldots ,N\}|/(2N+1)>0$.} contains arithmetic progressions of the form
$$
\{m, m+[a(n)],\ldots, m+\ell[a(n)]\}
$$
(see Theorem~\ref{T:SzSingle}). 
 Therefore, one can use any of the sequences  in \eqref{E:ExamplesConv} in place of $[a(n)]$ and also sequences like $[n^2+\log{n}]$ or $[\sqrt{2} \ \!n^2+\log\log n]$  (these sequences are bad for mean convergence).

   Furthermore, we study the averages \eqref{E:basic1} for sequences that are not necessarily in  arithmetic progression. We show that if $c_1,\ldots,c_\ell\in \R \setminus \Z$ are positive and distinct, then for every ergodic transformation $T$
we have
 \begin{equation}\label{E:basic3}
\lim_{N\to\infty}\frac{1}{N}\sum_{n=1}^N f_1(T^{[n^{c_1}]}x)\cdot\ldots\cdot f_\ell(T^{[n^{c_\ell}]}x)=
\int f_1 \ d\mu \cdot \ldots \cdot \int f_\ell \ d\mu
\end{equation}
where the convergence takes place in $L^2(\mu)$.
 More generally, one can replace the sequences $[n^{c_1}], \ldots, [n^{c_\ell}]$ in \eqref{E:basic3}
 with
 sequences
 $[a_1(n)],\ldots,[a_\ell(n)] $ where the functions  $a_1(t),\ldots,a_\ell(t)$  are  logarithmico-exponential
  and  satisfy some appropriate growth conditions
 (see Theorem~\ref{T:ConvSeveral}). This enables us to establish a conjecture of Bergelson and H{\aa}land-Knutson (\cite{BH09}, Conjecture~8.2).
We deduce  that if $c_1,\ldots,c_\ell\in \R \setminus \Z$ are positive, then
 every set of integers with positive upper density contains
  patterns of the form
$$
\{m, m+[n^{c_1}],\ldots, m+[n^{c_\ell}]\}.
$$

In the next section we give a more precise formulation of our main results
and also define  some of the concepts used throughout
the paper.
\section{Main results}
We first introduce some basic terminology needed to state our main results.
The reader will find more information related to the notions involved in the background section.

All along the article we will use the term {\it measure preserving
  system}, or the word {\it system}, to designate a quadruple
$(X,\X,\mu, T)$, where $(X,\X,\mu)$ is a Lebesgue
probability space, and $T\colon X\to X$ is an \emph{invertible}
measurable map such that $\mu(T^{-1}A)=\mu(A)$ for every
$A\in\X$.

Let $B$ be the collection of equivalence classes of real valued
  functions  defined on some half-line $(c,\infty)$, where we
  identify two functions if they agree eventually.\footnote{The
    equivalence classes just defined are often called \emph{germs of
      functions}. We are going to use the word function when we refer to
    elements of $B$ instead, with the understanding that all the
    operations defined and statements made for elements of $B$ are
    considered only for sufficiently large values of $t\in \R$.}  A
  \emph{Hardy field} is a subfield of the ring $(B,+,\cdot)$ that is
  closed under differentiation. With $\H$ we denote the \emph{union of all
  Hardy fields}. If $a\in \H$ is defined in $[1,\infty)$ (one can always
  choose such a representative of $a(t)$)
  we call the sequence $([a(n)])_{n\in\N}$ a  \emph{Hardy sequence}.
  Working within the class  $\H$  eliminates several
 technicalities  that would otherwise obscure the transparency of  our results and the main ideas of their proofs.
 Furthermore,  $\H$
 is a rich enough class to enable one to  deal, for example, with  all the sequences considered in \eqref{E:ExamplesConv}.

An explicit  example of a Hardy field to keep in mind  is the set
$\LE$ that consists of all  \emph{logarithmico-exponential
functions} (\cite{Ha10}, \cite{Ha12}), meaning  all
 functions defined on some  half-line $(c,\infty)$
using  a finite combination of the symbols $+,-,\times, :, \log, \exp$, operating on the real variable $t$
and on real constants.
For example, all rational functions and the functions
 $t^{\sqrt{2}}$,  $t\log{t}$,
$e^{\sqrt{\log \log t}}/\log(t^2+1)$ belong in $\LE$.

 The set $\H$ is much more extensive than the set  $\LE$, for example, one can show that it contains all antiderivatives of elements of $\LE$,  the Riemann zeta function $\zeta$, and  the Euler Gamma function $\Gamma$. Let us stress though that $\H$ does not contain functions that oscillate like $\sin{t}$ or $t\sin{t}$, or functions that have a derivative that oscillates, like $t^{100}+\sin{t}$.

 The most important  property of elements of $\H$ that will be used throughout this article is
 that we can relate  their growth rates with the growth rates  of their derivatives.

To simplify our exposition we introduce some notation. If  $a(t),b(t)$ are real valued functions defined on some half-line $(u,\infty)$ we write $a(t)\prec b(t)$  if $a(t)/b(t) \to 0$ as $t\to \infty$. (For example,  $1\prec \log{t}\prec t^\varepsilon$ for every $\varepsilon>0$.) We write $a(t)\ll b(t)$ if there exists $C\in \R$ such that
 $|a(t)|\leq C|b(t)|$ for all large enough $t\in \R$.
We say that a function $a(t)$ has \emph{polynomial growth}   if $a(t) \ll t^k$ for some $k\in \N$.

\subsection{Arithmetic progressions}\label{SS:ArithmeticProgressions} We are going to give
a collection of results that deal with
 multiple convergence and recurrence
properties of  Hardy sequences  of  polynomial growth.

\subsubsection{Convergence}\label{SSS:ApConvergence}
 Let $a(t)$ be a real valued function. We say that the sequence of integers $([a(n)])_{n\in\N}$ is \emph{good for
    multiple convergence} if for every $\ell \in \N$, system $(X,\X,\mu,T)$, and functions
  $f_1,f_2,\dots,f_\ell\in L^\infty(\mu)$, the averages
  \begin{equation}\label{E:basicAP}
    \frac1N \sum_{n=1}^N f_1(T^{[a(n)]}x)\cdot f_2(T^{2[a(n)]}x)\cdot \ldots
    \cdot f_\ell(T^{\ell [a(n)]}x)
  \end{equation}
  converge in $L^2(\mu)$  as $N\to\infty$.
As mentioned in the introduction, any polynomial with integer coefficients is an example of such a sequence.

The next result gives an extensive list of new examples of sequences that are good for multiple convergence.
In fact it shows that it is a rather rare occurrence for a
 Hardy sequence with polynomial growth to be bad for multiple convergence.
 \begin{theorem}\label{T:ConvSingle}
Let $a\in \H$ have polynomial growth.

 Then    the sequence  $([a(n)])_{n\in\N}$
is good for multiple convergence if and only if one of the following conditions holds:
\begin{itemize}
 \item   $|a(t)-cp(t)|\succ \log t$ for every $c\in\R$ and every  $p\in \Z[t]$; \text{ or }

 \item $a(t)-cp(t)\to d$ for some $c,d\in\R$ and some $p\in \Z[t]$; \text{ or }

 \item $|a(t)-t/m|\ll \log{t}$ for some $m\in\Z$.
\end{itemize}
\end{theorem}




\begin{remarks}
$\bullet$ It follows  that the sequences in  \eqref{E:ExamplesConv}
and the sequences $[\sqrt{5}n^2]$, $[n/2+\log{n}]$ are good for
multiple convergence. The sequences $[\sqrt{5}n^2+\log{n}]$,
$[2n+\log{n}]$ are not good.

$\bullet$ The same necessary and sufficient conditions for convergence of  ``single" ergodic averages
$\frac{1}{N} \sum_{n=1}^Nf(T^{[a(n)]}x)$ where previously established in \cite{BKQW05}.


$\bullet$ If $a(t)$ is a real valued polynomial, then  our argument  shows that the averages \eqref{E:basicAP} converge in $L^2(\mu)$ even
if one replaces
 the limit  $\lim_{N\to\infty}\frac1N \sum_{n=1}^N$ with
 the limit $\lim_{N-M\to\infty}\frac{1}{N-M} \sum_{n=M+1}^N$. On the other hand, if $a\in \H$
  satisfies $t^{k-1}\prec a(t)
 \prec t^k$ for some $k\in \N$, then one can show that  the sequence $([a(n)])_{n\in\N}$ takes odd (respectively even) values in arbitrarily long intervals; as a result the limit
$\lim_{N-M\to\infty}\frac{1}{N-M} \sum_{n=M+1}^N T^{[a(n)]}f$ does not exist in general.
\end{remarks}
The first condition of Theorem~\ref{T:ConvSingle} is satisfied by
``typical'' functions in $\H$ with polynomial growth. For such
functions, the next result allows us to identify the limit of the
averages \eqref{E:basicAP}:
\begin{theorem}\label{T:ConvSingleFormula}
Let $a\in \H$ have polynomial growth and satisfy $|a(t)-cp(t)|\succ \log t$ for every
$c\in \R$ and every $p\in \Z[t]$.

Then for every $\ell\in \N$, system $(X,\X,\mu,T)$, and functions $f_1,\ldots,f_\ell\in L^\infty(\mu)$,
 we have
\begin{equation}\label{E:Furste}
\lim_{N\to\infty}\frac{1}{N}\sum_{n=1}^N T^{[a(n)]}f_1\cdot T^{2[a(n)]}f_2\cdot \ldots\cdot T^{\ell[a(n)]}f_\ell=
\lim_{N\to\infty}\frac{1}{N}\sum_{n=1}^N T^{n}f_1\cdot T^{2n}f_2\cdot\ldots\cdot T^{\ell n}f_\ell
\end{equation}
where the limit is taken in $L^2(\mu)$.
\end{theorem}
\begin{remarks}
$\bullet$ Examples of Hardy sequences for which this result applies are those given in \eqref{E:ExamplesConv}.

$\bullet$ A rather explicit formula for the limit in \eqref{E:Furste} can be given by combining  results in \cite{HK05a} (see also \cite{Zi07}) and  \cite{Zi05}.

$\bullet$ If $a(t)=cp(t)+d$  for some $c\in \R $ and $p\in \Z[t]$, then  \eqref{E:Furste} typically fails.
One can see this by considering appropriate rotations on the circle and taking $a(t)=2t, t^2,$ or $\sqrt{2} t$.
\end{remarks}

\subsubsection{Recurrence}\label{SSS:ApRecurrence}
 Let $a(t)$ be a real valued function.
 We say that the sequence of integers $([a(n)])_{n\in\N}$ is \emph{good for multiple recurrence}
  if for every $\ell\in \N$, system   $(X,\X,\mu,T)$, and
  set $A\in \X$ with $\mu(A)>0$ we have
  \begin{equation}\label{E:cnm}
  \mu(A\cap T^{-[a(n)]}A\cap T^{-2[a(n)]}A\cap\cdots \cap T^{-\ell [a(n)]}A)>0
  \end{equation}
for some $n\in \N$ such that  $[a(n)]\neq 0$. One can check that if the sequence $([a(n)])_{n\in\N}$ is good for multiple recurrence, then  $\eqref{E:cnm}$ is satisfied for infinitely many $n\in \N$.

Let us discuss briefly the recurrence properties of   sequences  defined using polynomials with   real coefficients.
If $q\in \R[t]$ is  non-constant and has zero constant term,  then the sequence $q(n)$ is good for multiple recurrence (this follows from  \cite{BL96} and a trick used in \cite{BH96}).
 If  $q\in \Z[t]$ does not have zero constant term, then the sequence $q(n)$  is good for multiple recurrence if and only if   the range of the polynomial contains multiples of every positive integer (\cite{Fr08}).
 More generally, if $q\in \R[t]$,  then
 $[q(n)]$ is good for multiple recurrence unless $q(t)$ has the form $q(t)=cp(t)+d$ for some $p\in \Z[t]$
 and $c,d\in \R$ (one way to see this is to use  Theorem~\ref{T:RecSingle} below). In this last case deciding whether the sequence $[q(n)]$ is good for multiple recurrence is more delicate and depends on intrinsic properties of the polynomial $q$.
 For example, one can show that the sequences $[\sqrt{5}n+1]$ and  $[\sqrt{5}n+3]$
 are good for multiple  recurrence, but the sequence $[\sqrt{5}n+2]$
is bad for multiple recurrence.\footnote{ The sequence $[\sqrt{5}n+2]$ is bad for recurrence  because
 $\norm{[\sqrt{5}n+2]/\sqrt{5}}\geq 1/10$ for every $n\in\N$, where $\norm{\cdot}$ denotes the distance to the closest integer.
It can be shown (\cite{BH96}) that the sequence $[an+b]$, $a,b\in\R$, is good for
  single recurrence (meaning \eqref{E:cnm} holds for $\ell=1$) if and only if there exists an integer $k$ such that
  $ak+b\in [0,1]$ (this is equivalent to $\{b/a\}\leq 1/a$).  For other sequences of the form $[ap(n)+n]$, like  $[an^2+b]$, necessary and sufficient conditions seem to be more complicated.}

  Our next result shows that  if one avoids polynomial sequences, then
   every Hardy  sequence of polynomial growth is good for multiple recurrence:
\begin{theorem}\label{T:RecSingle}
Let $a\in \H$ have polynomial growth and suppose that  $a(t)-cp(t)\to \infty$ for every $c\in \R$ and $p\in \Z[t]$.

Then the sequence $([a(n)])_{n\in\N}$ is good for multiple recurrence.
\end{theorem}
\begin{remarks}
$\bullet$  Examples of Hardy sequences for which this result applies are given in \eqref{E:ExamplesConv}.
It also applies to the sequences $[\sqrt{5}n+\log{n}]$ and  $[n^2 +\log\log{n}]$.

$\bullet$ Theorem~\ref{T:RecSingle}  was previously established in  \cite{FrW09}
under a somewhat more restrictive assumption (namely $t^{k-1}\prec a(t)\prec t^k$ for some $k\in \N$). Furthermore,
 the single recurrence case   was previously established by Boshernitzan (unpublished), and subsequently in \cite{FrW09}.

$\bullet$ Let $R$  be the set  of those $n\in\N$
for which \eqref{E:cnm} holds.
Combining the multiple recurrence result of Furstenberg (\cite{Fu77}) and Theorem~\ref{T:ConvSingleFormula}, one sees that if $a(t)-cp(t)\succ\log{t}$  for every $c\in \R$ and $p\in \Z[t]$, then the set $R$ has positive lower density.
Unlike the case where $a(t)$ is polynomial,  if
$a\in \H$ satisfies $t^{k-1}\prec a(t)\prec t^{k}$ for some $k\in \N$, then
one can show that the sequence $[a(n)]$   takes odd values in arbitrarily long intervals, and as a result
for some systems  the set $R$ has unbounded gaps.
\end{remarks}

\subsubsection{Characteristic factors}\label{SSS:ApCharacteristic}
Let  $(X,\X,\mu,T)$ be a system. A
 factor $\mathcal{C}$ is called a \emph{characteristic factor}, or \emph{characteristic}, for the family  of integer sequences $\{a_1(n),\ldots, a_\ell(n)\}$, if whenever
  one of the functions $f_1,\ldots,f_\ell\in L^\infty(\mu)$ is orthogonal to $\mathcal{C}$, the averages
\begin{equation}\label{E:fvv}
\frac{1}{N}\sum_{n=1}^N f_1(T^{a_1(n)}x)\cdot \ldots\cdot f_\ell(T^{a_\ell(n)}x)
\end{equation}
converge to $0$ in $L^2(\mu)$ as $N\to\infty$.

It follows  that if $\mathcal{C}$ is as above, then the limiting behavior of the  averages \eqref{E:fvv} remains unchanged if one projects each function to the factor $\mathcal{C}$, meaning that
the difference of the two averages converges to $0$ in $L^2(\mu)$ as $N\to\infty$.

It is known that the nilfactor $\mathcal{Z}$ of a system (defined in Section~\ref{SS:BackgroundErgodic}) is characteristic for every family $\{p(n), 2p(n),\ldots, \ell p(n)\}$ whenever  $p$ is an integer polynomial (\cite{HK05b}). We extend this result by showing:

\begin{theorem}\label{T:CharSingle}
Suppose that  $a\in \H$ has polynomial growth and satisfies $a(t)\succ \log{t}$.

Then for every  system and $\ell \in \N$, the  nilfactor $\mathcal{Z}$  of the system is characteristic for the family
 $\{[a(n)], 2[a(n)],\ldots, \ell [a(n)]\}$, for every $\ell\in \N$.
\end{theorem}


\begin{remarks}
$\bullet$ If $a(t)\ll \log{t}$, then the result fails even for $\ell=1$, the reason being that
the sequence $[a(n)]$ remains constant on  some sub-interval of $[1,N]$ that has  length  proportional
to $N$ as $N\to \infty$. Therefore,  if the transformation $T$ is  weakly mixing  but not strongly mixing,
the function $f$ has  zero integral  and satisfies $\int f \cdot T^{n}f\ d\mu\not\to 0$, then  $f\bot \cZ$ but
$\frac{1}{N}\sum_{n=1}^N  T^{[a(n)]}f\not\to 0$.

$\bullet$ A related result was proved in \cite{BH09} for weakly mixing systems assuming that the function $a(t)$ is tempered (for $a\in \H$ this is equivalent to  $t^{k-1}\log t\prec a(t)\prec t^{k}$ for some $k\in\N$). Since the method used in \cite{BH09} does not  work for functions like $t^k\log{t}$, we will
use a  different approach to prove Theorem~\ref{T:CharSingle}.
\end{remarks}

\subsubsection{Combinatorics}\label{SSS:ApCombinatorics}
Using  the previous multiple recurrence result we derive   a refinement of Szemer\'edi's Theorem on arithmetic progressions.
We will use the following correspondence principle of Furstenberg  (the formulation given is from \cite{Be87b}):
\begin{Correspondence1}Let $\Lambda$ be a
  set of integers.

  Then  there exist a system $(X,\B,\mu,T)$ and a set
  $A\in\X$, with $\mu(A)=\bar{d}(\Lambda)$, and such that
  \begin{equation}\label{E:correspondence}
    \bar{d}(\Lambda \cap (\Lambda-n_1)\cap\ldots\cap
    (\Lambda-n_\ell))\geq \mu(A\cap T^{-n_1}A\cap\cdots \cap T^{-n_\ell}A)
  \end{equation}
  for every $n_1,\ldots,n_\ell\in\Z$ and $\ell\in\N$.
\end{Correspondence1}

Using the previous principle and Theorem~\ref{T:RecSingle}
we immediately deduce the following:
\begin{theorem}\label{T:SzSingle}
Let $a\in \H$ have polynomial growth and suppose that  $a(t)-cp(t)\to \infty$ for every $c\in \R$ and every $p\in \Z[t]$.

Then for every $\ell \in\N$,  every $\Lambda\subset\mathbb{Z}$ with
  $\bar{d}(\Lambda)>0$ contains arithmetic progressions of the form
  \begin{equation}\label{E:SzSingle}
    \{m, m+[a(n)],m+2[a(n)],\ldots, m+\ell[a(n)]\}
  \end{equation}
  for some $m\in \Z $ and $n\in\N$ with $[a(n)]\neq 0$.
\end{theorem}

\subsubsection{More general classes of functions}\label{SSS:ApGeneral}
We make some remarks about the extend of the functions our methods cover that do not necessarily belong to some Hardy field.

 The conclusions of Theorems~\ref{T:ConvSingleFormula}, \ref{T:RecSingle}, \ref{T:CharSingle}, and \ref{T:SzSingle} hold if for some
$k\in \N$ the function $a\in C^{k+1}(\R_+)$ satisfies
$$
|a^{(k+1)}(t)| \text{ decreases to zero}, \quad 1/t^k\prec a^{(k)}(t)\prec 1,
\quad  \text{and } \quad  (a^{(k+1)}(t))^k\prec (a^{(k)}(t))^{k+1}.
$$
(If $a\in \H$, these three conditions are equivalent to
``$a(t)$ has polynomial growth and $|a(t)-p(t)|\succ \log{t}$ for every $p\in \R[t]$''.)
One can see this by  repeating verbatim the proofs  given in this article and in \cite{Fr09}. The reader is advised to think
of the second condition
as the most important one and the other two as  technical necessities (for functions in $\H$ the second condition implies the other two).

As for Theorem~\ref{T:ConvSingle},  unless one works within
 a ``regular" class of functions  like $\H$, it seems impossible to get explicit necessary and sufficient
 conditions.

\subsection{Several sequences}\label{SS:SeveralSequences} We are going to give  results related to multiple convergence and recurrence
properties involving several sequences  of polynomial growth.
For practical reasons (mainly expository) we are going to restrict ourselves  to the case where all the functions involved
are logarithmico-exponential. More technically involved arguments
should  enable one to extend the results
mentioned below to the case where all the functions belong to the same Hardy field.

Let us also remark that the results we give below
 are certainly less exhaustive than the results of Section~\ref{SS:ArithmeticProgressions}.
We are able to handle  a case that includes all functions given by fractional powers of $t$
and is general enough to cover a conjecture of Bergelson and H{\aa}land.
The expected ``optimal'' results involving several sequences are stated
in Problems 2, 3, and 4 of Section~\ref{SS:Conjectures}.
\subsubsection{Convergence}\label{SSS:SevConvergence}
To simplify our statements we introduce the following class of ``good'' (for our purposes) functions:
\begin{equation}\label{E:G}
\G=\{a\in C(\R_+)\colon t^{k+\varepsilon}\prec a(t)\prec t^{k+1} \text{ for some integer } k\geq 0 \text{ and  some } \varepsilon>0\}.
\end{equation}
  Equivalently, a function $a\in \H$  with polynomial growth  belongs in $\G$ \emph{unless} for some  integer $k\geq 0$  we have   $t^k\prec a(t)\prec t^{k+\varepsilon}$ for
every $\varepsilon>0$.
For example,  if $c\geq 0$, then  $t^c\in \G$ if and only if $c$ is not an integer. The reader is advised to think of functions in $\G$ as having ``fractional-power growth rate''.

The next result (in fact its corollary Theorem~\ref{T:RecMult}) verifies a conjecture of Bergelson and H{\aa}land-Knutson (\cite{BH09}, Conjecture~8.2).\footnote{A comment about notation. In \cite{BH09} the class $\LE$ is denoted by $\H$. Also, for functions in $\LE$ our class $\G$ coincides with the class $\mathcal{T}$ defined in \cite{BH09}.} It
 shows that
for ``typical" logarithmico-exponential functions  of polynomial growth  the limit of the averages \eqref{E:basic1} exists and for ergodic systems it is constant.  We say that the functions $a_1(t),\ldots,a_\ell(t)$  have \emph{different growth rates} if the quotient of any two of these functions  converges  to $\pm \infty$ or to $0$.
\begin{theorem}\label{T:ConvSeveral}
Suppose that the functions $a_1,\ldots,a_\ell \in\mathcal{LE}\cap \G$ have different growth rates.

Then  for every ergodic system $(X,\mathcal{B},\mu,T)$
 and
   $f_1,\dots,f_\ell\in L^\infty(\mu)$ we have
  \begin{equation}\label{E:product}
\lim_{N\to\infty}    \frac1N \sum_{n=1}^N f_1(T^{[a_1(n)]}x)\cdot  \ldots
    \cdot f_\ell(T^{ [a_\ell(n)]}x)=\int f_1\ d\mu \cdot\ldots \cdot \int f_\ell \ d\mu
  \end{equation}
where the convergence takes place in $L^2(\mu)$.
\end{theorem}
\begin{remarks}
$\bullet$ Some  examples for which our result applies are given by the collections of sequences
$\{[n^{1/2}],[n^{3/2}], [n^{5/2}]\}$, and  $\{[n^{\sqrt{2}}], [n^{\sqrt{2}}\log\log{n}],[n^{\sqrt{2}}\log{n}] \}$.

$\bullet$ Equation \eqref{E:product}  fails
for some ergodic systems if a non-trivial linear combination  of the functions $a_1(t),\ldots,a_\ell(t)$ is an  integer polynomial other than $\pm t+k$.


$\bullet$ A substantial part of the proof (carried out in the companion paper \cite{Fr09}) is consumed in working on a potentially  non-trivial (characteristic) factor  of our system.  Initially we show that this factor  has (roughly speaking)  the structure of a nilsystem, only to realize later (using some non-trivial equidistribution results   on nilmanifolds) that this factor is trivial.
 It would be nice
 to have a proof  that avoids such diversions to  non-Abelian analysis.
\end{remarks}




If all the functions $a_1(t),\ldots,a_\ell(t)$ have sub-linear growth then
Theorem~\ref{T:ConvSeveral} can be (rather easily)  proved in a more general setup, where one uses iterates of $\ell$ not necessarily commuting ergodic transformations in place of a single ergodic  transformation.

\begin{theorem}\label{T:ConvNonCommuting}
Let $a_1,\ldots,a_\ell \in\mathcal{LE}\cap \mathcal{G}$ have different  growth rates and satisfy
$a_i(t)\prec t$ for $i=1,\ldots,\ell$.
Let  $T_1,\ldots,T_\ell$ be invertible measure preserving transformations acting on a probability space $(X,\X,\mu)$.

Then  for every
   $f_1,\dots,f_\ell\in L^\infty(\mu)$ we have
  $$
\lim_{N\to\infty}    \frac1N \sum_{n=1}^N f_1(T^{[a_1(n)]}_1x)\cdot  \ldots
    \cdot f_\ell(T^{ [a_\ell(n)]}_\ell x)= \tilde{f}_1 \cdot\ldots \cdot \tilde{f}_\ell,
  $$
where $\tilde{f}_i=\E(f_i|\mathcal{I}(T_i))$, and the convergence takes place in $L^2(\mu)$.
\end{theorem}
\begin{remarks}
$\bullet$ In \cite{BH09} a similar result was proved for iterates of a single transformation.

$\bullet$ It is not known whether similar convergence results hold without any commutativity assumption on the transformations $T_i$ for some choice of functions $a_i(t)$ with different,  at least linear growth rates. On the other hand,  it is known (\cite{Bere85}) that
for some choice of non-commuting transformations $T_1,T_2$ and functions $f_1,f_2$, the averages  $\frac1N \sum_{n=1}^N f_1(T^{n}_1x)\cdot f_2(T^{n}_2 x)$
diverge in $L^2(\mu)$.
\end{remarks}

\subsubsection{Recurrence}\label{SSS:SevRecurrence}
The next multiple recurrence result is a consequence of Theorem~\ref{T:ConvSeveral}:
\begin{theorem}\label{T:RecMult}
Suppose that the functions $a_1,\ldots,a_\ell \in\mathcal{LE}\cap \G$ have different growth rates.

Then  for every system $(X,\X,\mu,T)$ and set $A\in \X$ we have
\begin{equation}\label{E:lowerbounds}
\lim_{N\to\infty} \frac{1}{N}\sum_{n=1}^N\mu(A\cap T^{-[a_1(n)]}A\cap T^{-[a_2(n)]}A\cap\cdots \cap
T^{-[a_\ell(n)]}A)\geq (\mu(A))^{\ell +1}.
\end{equation}
\end{theorem}
\begin{remarks}
$\bullet$ The estimate \eqref{E:lowerbounds} becomes an equality when the system is ergodic.

$\bullet$ The lower bounds \eqref{E:lowerbounds} contrast the corresponding lower bounds when the functions $a_1(t),...,a_\ell(t)$ are non-constant integer polynomials.  In this case, \eqref{E:lowerbounds}
fails even when $\ell=1$  and $a_1(t)=t^2$. In fact no power type lower bound is known for any
collection of polynomials (except of course  when all the functions are equal and  linear).
\end{remarks}
\subsubsection{Characteristic factors} \label{SSS:SevCharacteristic}
The next result gives convenient  characteristic factors  for a family of
 ``typical" logarithmico-exponential sequences  of polynomial growth.
\begin{theorem}\label{T:CharSeveral}
 Let $a_1,\ldots,a_\ell\in \mathcal{LE}$,  and suppose that
 all the functions $a_i(t)$ and their pairwise differences $a_i(t)-a_j(t)$
belong in $\G$ (defined \eqref{E:G}).

 Then for every system  its nilfactor $\cZ$ is characteristic for the family
 $\{[a_1(n)], \ldots, [a_\ell(n)]\}$.
 \end{theorem}
\begin{remark}
A related result was proved in \cite{BH09} for weakly mixing systems. In fact we are going to adapt
the argument  used in \cite{BH09} to establish our result.
\end{remark}
\subsubsection{Combinatorics}\label{SSS:SevCombinatorics}
Using Furstenberg's Correspondence Principle and Theorem~\ref{T:RecMult}
 we immediately deduce the following:

\begin{theorem}\label{T:SzMultiple}
Suppose that the functions $a_1,\ldots,a_\ell \in\mathcal{LE}\cap \G$ have different growth rates.

Then for  every set of integers $\Lambda$  we have
$$
\liminf_{N\to\infty}\frac{1}{N}\sum_{n=1}^N \overline{d}(\Lambda\cap (\Lambda-[a_1(n)])\cap\cdots\cap (\Lambda-[a_\ell(n)]))
\geq \big(\overline{d}(\Lambda)\big)^{\ell+1}.
$$
\end{theorem}






\subsection{Further directions}\label{SS:Conjectures}
We state some  open problems that are closely  related to the results stated before.
To avoid repetition we remark that \emph{in Problems 1-4  we always   work with a family  $\mathcal{F}=\{a_1(t),\ldots,a_\ell(t)\}$
  of functions of polynomial growth that belong to the same Hardy field}. With
  $\text{span}(\mathcal{F})$ we denote the set of all \emph{non-trivial} linear combinations of elements of $\mathcal{F}$.

\subsubsection{Convergence}\label{SSS:ConConvergence}
The family of functions $\mathcal{F}=\{a_1(t),\ldots,a_\ell(t)\}$ is \emph{good for multiple convergence}
if for every  system $(X,\X, \mu,T)$ and functions $f_1,\ldots,f_\ell\in L^\infty(\mu)$
the limit
\begin{equation}\label{E:fes}
\lim_{N\to\infty}    \frac1N \sum_{n=1}^N f_1(T^{[a_1(n)]}x)\cdot  \ldots
    \cdot f_\ell(T^{ [a_\ell(n)]}x)
    \end{equation}
exists in $L^2(\mu)$.

The next  problem
is much in the spirit of Theorem~\ref{T:ConvSingle}:
\begin{conjecture}\label{C:ConjConv}
 The family  $\mathcal{F}$ is good for multiple convergence
 if and only if  every  function  $a\in\text{span}(\mathcal{F})$ satisfies one of  the following conditions:
\begin{itemize}
 \item   $|a(t)-cp(t)|\succ \log t$ for every $c\in\R$ and every $p\in \Z[t]$; \text{ or }

 \item $a(t)-cp(t)\to d$ for some $c,d\in\R$; \text{ or }

 \item $|a(t)-t/m|\ll \log{t}$ for some $m\in\Z$.
\end{itemize}
\end{conjecture}

The next problem  provides a possible generalization of Theorem~\ref{T:ConvSeveral}:
\begin{conjecture}\label{C:ConjProduct}
Suppose
that every function  $a\in \text{span}(\mathcal{F})$ satisfies $|a(t)-cp(t)|\succ \log{t}$ for every $c\in \R$ and
every $p\in \Z[t]$.

Then  for every ergodic system $(X,\mathcal{B},\mu,T)$
 and
   $f_1,\dots,f_\ell\in L^\infty(\mu)$ we have
  \begin{equation}\label{E:product'}
\lim_{N\to\infty}    \frac1N \sum_{n=1}^N f_1(T^{[a_1(n)]}x)\cdot  \ldots
  \cdot f_\ell(T^{ [a_\ell(n)]}x)=\int f_1\ d\mu \cdot\ldots \cdot \int f_\ell \ d\mu
 \end{equation}
where the convergence takes place in $L^2(\mu)$.
\end{conjecture}
We remark that  if some function $a\in \text{span}(\mathcal{F})$ satisfies $|a(t)-cp(t)|\ll \log{t}$
for some $c\in \R$ and $p\in \Z[t]$ with $\deg(p)\geq 2$, then  \eqref{E:product'}
fails for some system.

\subsubsection{Characteristic factors}\label{SSS:ConCharacteristic}
We state a possible generalization  of Theorem~\ref{T:CharSeveral}:
\begin{conjecture}\label{C:ConjChar}
Suppose that
$a_i(t)\succ \log{t}$ and $a_i(t)-a_j(t)\succ \log{t}$ whenever $i\neq j$.

 Then for every system its nilfactor $\cZ$ is characteristic for the family
 $\{[a_1(n)], \ldots, [a_\ell(n)]\}$.
\end{conjecture}
One can easily see that the stated assumptions are also necessary.
\subsubsection{Recurrence}\label{SSS:Conrecurrence}
The next problem provides a possible extension of  Theorem~\ref{T:RecSingle}:
\begin{conjecture}\label{C:ConjRec}
Suppose that every function  $a\in \text{span}(\mathcal{F})$ satisfies $|a(t)-cp(t)|\to \infty$
for every  $c\in \R$ and every $p\in \Z[t]$.

Then for every  system $(X,\X,\mu,T)$ and $A\in \X$ with $\mu(A)>0$ we have
$$
\mu(A\cap T^{-[a_1(n)]}A\cap\cdots \cap T^{-[a_\ell(n)]}A)>0
$$
for some $n\in \N$ such that $[a_i(n)]\neq 0$.
\end{conjecture}
An interesting  special case  of this result  is when the
functions $a_1(t),\ldots,a_\ell(t)$ have different growth and none of them is equal to a
 polynomial (modulo a function that vanishes at infinity).

If all the functions $a_1(t),\ldots,a_\ell(t)$ are integer polynomials, then necessary and sufficient conditions
for multiple recurrence where given in \cite{BLL08}.
\subsubsection{Combinatorics}\label{SSS:ConCombinatorics}
We rephrase  Problem~\ref{C:ConjRec} in combinatorial terminology:
\begin{conjecture2'}\label{C:ConjRec'}
Suppose that every function  $a\in \text{span}(\mathcal{F})$ satisfies $|a(t)-cp(t)|\to \infty$
for every  $c\in \R$ and every $p\in \Z[t]$.

Then  every $\Lambda\subset\mathbb{Z}$ with
  $\bar{d}(\Lambda)>0$ contains patterns of the form
 \begin{equation}\label{E:SzMultiple}
    \{m, m+[a_1(n)],\ldots, m+[a_\ell(n)]\}
  \end{equation}
  for some $m\in \Z $ and $n\in\N$ with $[a_i(n)]\neq 0$.
\end{conjecture2'}

\subsubsection{Commuting transformations}\label{SSS:ConCommuting}
It seems likely that our main results
remain true when one works  with iterates of $\ell$ commuting measure preserving transformations instead of iterates a single transformation. We  state two  related problems here:

\begin{conjecture}\label{C:ConjCommuting1}
Let  $T_1,\ldots,T_\ell$ be commuting invertible measure preserving transformations acting on a probability space $(X,\X,\mu)$ and $f_1,\ldots,f_\ell\in L^\infty(\mu)$.

 Then  for every positive  real number $c$ the following limit exists in $L^2(\mu)$
 \begin{equation}\label{E:abra}
\lim_{N\to\infty}  \frac1N \sum_{n=1}^N f_1(T^{[n^c]}_1x)\cdot \ldots\cdot f_\ell(T^{[n^c]}_\ell x).
 \end{equation}
 Furthermore,   if $c$ is not an integer, then  \eqref{E:abra} is equal to
 $\lim_{N\to\infty}  \frac1N \sum_{n=1}^N T^{n}_1 f_1\cdot \ldots\cdot T^{n}_\ell f_\ell $.
\end{conjecture}
For $c=1$ the existence of the limit \eqref{E:abra}  was established
by Tao in \cite{Ta08} (see also \cite{To09}, \cite{Au09},
\cite{Ho09} for other subsequent proofs). The case $0<c<1$ can be
easily reduced to the case $c=1$ (see Lemma~\ref{L:ChangeVar}
below).

\begin{conjecture}\label{C:ConjCommuting2}
Let $T_1,\ldots,T_\ell$ be commuting measure preserving transformations acting on a probability space $(X,\X,\mu)$.
Let $c_1,\ldots,c_\ell\in \R\setminus \Z$ be  positive and distinct.

Then  for every
   $f_1,\dots,f_\ell\in L^\infty(\mu)$ we have
  $$
\lim_{N\to\infty}    \frac1N \sum_{n=1}^N f_1(T^{[n^{c_1}]}_1x)\cdot  \ldots
    \cdot f_\ell(T^{ [n^{c_\ell}]}_\ell x)= \tilde{f}_1 \cdot\ldots \cdot \tilde{f}_\ell,
  $$
where $\tilde{f}_i=\E(f_i|\mathcal{I}(T_i))$, and the convergence takes place in $L^2(\mu)$.
\end{conjecture}
This is an immediate consequence of Theorem~\ref{T:ConvNonCommuting}
when all the exponents $c_i$ are smaller than $1$  and the commutativity of the transformations $T_i$ is not needed in this case.

\subsubsection{Prime numbers}\label{SSS:ConPrimes}
The results of this article  should remain true when one makes the substitution $n\to p_n$ where
$p_n$ denotes the $n$-th prime number.
For instance:
\begin{conjecture}
 If $c$ is a positive non-integral real number,
    then the sequence $[p_n^c]$ is good for multiple recurrence and convergence.
Furthermore, the limit of the  corresponding multiple ergodic averages is equal to the limit of
the ``Furstenberg averages" (defined by \eqref{E:Furst}).
\end{conjecture}
 One could try to verify
such a statement by comparing  the averages along $[p_n^c]$
to the averages along the sequence $[n^{c}]$  for which all required properties are known.
A similar strategy was used  in \cite{FrHK07} to deal with double recurrence (and convergence) problems of the shifted primes.

Another challenge is to  use the Szemer\'edi type results of
Sections~\ref{SS:ArithmeticProgressions} and \ref{SS:SeveralSequences}
and prove that the primes contain the corresponding Hardy-field patterns. For instance:
 \begin{conjecture}\label{Pr:8}
If $c,c_1,c_2\in \R$ are positive, then   the  prime numbers  contain patterns of the form
 $$
 \{m,m+[n^{c}],m+2[n^{c}]\} \ \text{ and } \
\{m,m+[n^{c_1}],m+[n^{c_2}]\}.
$$
\end{conjecture}
We remark that using the corresponding density results (in addition to many other things), Green and Tao (\cite{GT08}) proved the existence of arbitrarily  long arithmetic progressions in the primes,
and Tao and Ziegler  (\cite{TZ08}) the existence of arbitrarily  long polynomial progressions in the primes (this last result allows one to handle  Problem~\ref{Pr:8} when  $c,c_1,c_2$ are positive rational numbers).
\subsection{Structure of the article and main ideas}
In Section~\ref{S:background} we gather some essential background material from ergodic theory,
equidistribution results on nilmanifolds, and basic facts about Hardy fields. Key for our study is
the structure theorem of Host and Kra (Theorem~\ref{T:HoKra}) and the quantitative equidistribution result of Green and Tao (Theorem~\ref{T:GT2}). The use of the latter result is rather implicit in this article since we frequently use results from the  companion article \cite{Fr09} that were proved using quantitative equidistribution.

In Section~\ref{S:CharFactorsSingle} we prove Theorem~\ref{T:CharSingle} which shows that under appropriate conditions
the nilfactor
is characteristic for families of the form $\{[a(n)],\ldots,\ell[a(n)]\}$. We remark
 that for functions of ``fractional-power'' growth (like $a(t)=t^{3/2}$), this problem can be handled using more or
 less conventional techniques. But for functions that have slowly growing derivatives  (like $a(t)=t\log{t}$)  the
 ``standard'' techniques become problematic.
 To overcome this problem,  we partition the positive integers into intervals of appropriate
size, and in each such interval we use the Taylor expansion of the function to get an approximation  by real valued  polynomials of fixed degree. This approximation works well when the function stays logarithmically away from
polynomials, and as a result functions like $t^{3/2}$, $t\log{t}$, and $t+(\log{t})^2$ become practically indistinguishable for our purposes.
After performing these initial maneuvers we are led to estimating some multiple ergodic averages involving  polynomial
iterates (Proposition~\ref{P:uniform}), a problem that can be handled using more or less standard techniques.

In Section~\ref{S:CharFactorSeveral} we prove Theorem~\ref{T:CharSeveral} which shows that under appropriate conditions the nilfactor
is  characteristic for families of the form  $\{[a_1(n)],\ldots,[a_\ell(n)]\}$. Since we only work
 with functions of ``fractional-power'' growth,
 we are able to adapt an  argument of Bergelson and H{\aa}land  (\cite{BH09}) that was used to
 establish a convergence result for  weakly mixing systems.
 The proof consists of two steps. One first deals with the case where all the functions have at most linear growth (Proposition~\ref{P:CharSublin}). This is done  by successively applying  Van der Corput's lemma and a change of variable trick.
Then one uses   a modification of the polynomial exhaustion technique of Bergelson
to reduce the general case to the case of at most linear growth.

In the last section we complete the proof of  the convergence and recurrence results of Sections~\ref{SS:ArithmeticProgressions} and \ref{SS:SeveralSequences}.
With the exception of Theorem~\ref{T:ConvNonCommuting} that can be handled directly, to prove the other convergence results
 we first make use of the
 results from Sections~\ref{S:CharFactorsSingle} and \ref{S:CharFactorSeveral}
to show that the nilfactor of the system is characteristic for the appropriate multiple ergodic averages.
Then Theorem~\ref{T:HoKra} enables us to reduce matters to nilsystems. Finally, we
use equidistribution results from the companion paper  \cite{Fr09} to verify the appropriate
convergence property for nilsystems. The recurrence results
are direct consequences of the corresponding convergence results, with the exception of
a special case of  Theorem~\ref{T:RecSingle} where the function  is logarithmically close to
a constant multiple of an integer polynomial.
In this case, a somewhat
complicated analysis is used to prove the corresponding recurrence property for nilsystems.

\subsection{Notational conventions.} The following notation will be
used throughout the article: $\N=\{1,2,\ldots\}$, $\T^k=\R^k/\Z^k$, $Tf=f\circ T$,
$e(t)=e^{2\pi i t}$,
 $[t]$ denotes the integer part of $t$,  $\{t\}=t-[t]$,
 $\norm{x}=d(t,\Z)$,
$\E_{n\in A} a(n)=\frac{1}{|A|}\sum_{n\in A} a(n)$.  We sometimes write $t$ to represent an element $t\Z$ of $\T=\R/\Z$.
By $a(t)\prec b(t)$ we mean
$\lim_{t\to\infty}a(t)/b(t)=0$, by $a(t)\sim b(t)$ we mean $\lim_{t\to\infty}a(t)/b(t)$ is a non-zero real number,
and by $a(t)\ll b(t)$ we mean $|a(t)|\leq C |b(t)|$
 for some absolute constant $C$ and all large enough values of $C$ . By $\R_+$ we denote some half-line $(c,+\infty)$.
When often write $\infty$ instead of $+\infty$. By $\mathcal{H}$ we
denote the class of all functions that belong to some Hardy field,
by $\mathcal{LE}$ the class of all logarithmic-exponential
functions, and by $\mathcal{G}$ the class of functions in $C(\R_+)$
that satisfy $t^{k+\varepsilon}\prec a(t)\prec t^{k+1}$ for some
non-negative integer $k$ and $\varepsilon>0$.

\subsection{Acknowledgement.} The author would like to thank
J.~Ku{\l}aga-Przymus for pointing out a mistake in the statement of
Lemma~\ref{L:ChangeVar}.

\section{Background material}\label{S:background}
\subsection{Hardy fields}\label{SS:Hardy}
We collect here  some basic properties of elements of $\H$ relevant to our study.
The reader can find more background material in \cite{Bos94} and the references
therein.

Every element of $\H$ has eventually constant sign.
 Therefore, if  $a\in \H$, then $a(t)$ is eventually monotonic (since $a'(t)$ has eventually constant sign), and  the limit
$\lim_{t\to \infty} a(t)$ exists (possibly infinite).  For every two
functions $a\in \H, b\in \LE$ $(b\neq 0)$, we have $a/b\in \H$. It follows that   the asymptotic growth ratio
$\lim_{t\to \infty}a(t)/b(t)$ exists (possibly infinite).

We  caution the reader  that $\H$ is not a field, and some pairs of functions in $\H$  are not
asymptotically comparable. This defect of
 $\H$  plays a role in some of our results, and can be sidestepped  by  restricting
 our attention to
 the Hardy field of logarithmic-exponential functions $\LE$.

A key property of elements of $\H$ with
 polynomial growth is that one can relate their growth rates with
the growth rates of their derivatives:
\begin{lemma}[{\bf Lemma 2.1 in  \cite{Fr09}}]\label{L:properties}
Suppose that   $a\in \H$ has polynomial growth.  We have the following

$(i)$ If $t^\varepsilon \prec a(t)$ for some $\varepsilon>0$, then  $a'(t)\sim a(t)/t$.

$(ii)$ If $t^{-k}\prec a(t)$ for some $k\in\N$, and $a(t)$ does not converge to a non-zero constant,
then  $a(t)/(t(\log{t})^2)\prec a'(t)\ll a(t)/t$.
\end{lemma}
We are going to freely use all these properties in the sequel.

\subsection{Nilmanifolds}
\label{SS:BackgroundNil}
 All the equidistribution results  needed in
this article
were established
in \cite{Fr09} with the exception of one result that will be needed
to cover a  special case of  Theorem~\ref{T:RecSingle}.
 Below we gather some basic facts and a quantitative equidistribution result that
  will be used in its proof.

  The proofs of all the results mentioned below
can be found or deduced from \cite{Lei05a} and \cite{GT09c}.

\subsubsection{Definitions and basic properties}
A \emph{nilmanifold}  is a homogeneous space $X=G/\Gamma$ where  $G$ is a nilpotent Lie group,
and $\Gamma$ is a discrete cocompact subgroup of $G$. If $G_{k+1}=\{e\}$ , where $G_k$ denotes the $k$-th commutator subgroup of $G$, we say that $X$ is a
 $k$-\emph{step} nilmanifold.  With $G_0$ we denote the connected component of the identity element in $G$.
 The representation of a nilmanifold
$X$ as a homogeneous
space of a nilpotent Lie group $G$ is not unique. It can be shown (\cite{Lei05a}) that if $X$
is connected, then it  admits a
representation of the form $X=G/\Gamma$ such that
$G_0$ is simply connected  and  $G=G_0\Gamma$.
{\em For connected nilmanifolds $X=G/\Gamma$,  we will always assume that $G$ satisfies these two
extra assumptions.}

The group $G$ acts on $G/\gG$ by left
translation where the translation by a fixed element $b\in G$ is given
by $T_{b}(g\gG) = (bg) \gG$.  By $m_X$ we denote the unique probability
measure on $X$ that is invariant under the action of $G$ by left
translations (called the {\it normalized Haar measure}) and $\G/\gG$ denote the
Borel $\sigma$-algebra of $G/\gG$. Fixing an element $b\in G$, we call
the system $(G/\gG, \G/\gG, m, T_{b})$ a {\it  nilsystem}. We
call the elements of $G$ \emph{nilrotations}.

 For every $b\in G$ the  set $X_b=\overline{\{b^n\Gamma, n\in\N\}}$   is a nilmanifold $H/\Delta$, where $H$ is a closed subgroup of $G$ that contains $b$, and $\Delta=H\cap \Gamma$
is a discrete cocompact subgroup of $H$. Furthermore, for every $b\in G$ there exists an $r\in \N$ such that the nilmanifold $X_{b^r}$ is connected.

A nilrotation $b\in G$ is \emph{ergodic},
or \emph{acts ergodically on $X$},  if the sequence $(b^n\Gamma)_{n\in\N}$ is dense
in $X$. If $b\in G$ is ergodic, then  for
every $x\in X$ the sequence $(b^nx)_{n\in\N}$ is equidistributed in $X$.
  If the nilmanifold $X$ is connected and $b$ acts ergodically on $X$, then
 for every $r\in \N$ the element  $b^r$ also acts ergodically on $X$.


\subsubsection{A quantitative equidistribution result}
If $G$ is a nilpotent group, then a sequence $g\colon \Z\to G$ of the
form $g(n)=b_1^{p_1(n)}\cdot\ldots\cdot b_k^{p_k(n)}$, where $b_i\in G$,
and $p_i$ are polynomials taking integer values at the integers, is
called a \emph{polynomial sequence in} $G$. If the maximum
degree of the polynomials $p_i$ is at most $d$ we say that the
\emph{degree} of $g(n)$ is at most $d$.
A \emph{polynomial sequence on the
  nilmanifold} $X=G/\Gamma$ is a sequence of the form
$(g(n)\Gamma)_{n\in\Z}$ where $g\colon \Z\to G$ is a polynomial
sequence in $G$.

In \cite{GT09c}, Green and Tao proved a quantitative equidistribution result
for  polynomial sequences on nilmanifolds $X=G/\Gamma$ when the group
 $G$ is connected and simply connected.
We will need an  extension of this result to the non-connected case.
In order to state it we first introduce  some notions from \cite{GT09c} and \cite{FrW09}.

If $X=G/\Gamma$ is a connected nilmanifold, the {\em affine torus} of
$X$ is defined to be the homogeneous space $A=G/([G_0,G_0]\Gamma)$.
It is known (\cite{FrK06}) that every  nilrotation acting on
the affine torus  is isomorphic to a  unipotent affine transformation
on
  some finite dimensional torus\footnote{This means  $T\colon \T^l\to \T^l$ has the form
$T(t) = b\cdot S(t)$ for some unipotent homomorphism $S$ of $\T^l$
and $b\in \T^l$.} with the normalized Haar measure, and  furthermore
  the conjugation map can be taken to be continuous.
We can therefore identify the affine torus $A$ of a
nilmanifold $X$ with a finite dimensional torus $\T^l$ and think of a
nilrotation acting on $A$ as a unipotent affine transformation on
$\T^l$.

A \emph{quasi-character} of a nilmanifold $X=G/\Gamma$ is a
function $\psi\colon G\to \C$
that is a continuous homomorphism of $G_0$ (to the multiplicative group $\{z\in \C\colon |z|=1\}$) and satisfies $\psi(g\gamma)=\psi(g)$
for every $\gamma\in\Gamma$. Every quasi-character annihilates
$[G_0,G_0]$, and as a result factors through the affine torus $A$ of $X$. Under
an appropriate  isomorphism we have that $A\simeq
\T^l$ and every quasi-character of $X$ is mapped to a character of
$\T^l$. Therefore, thinking of $\psi$ as a character of $\T^l$ we have
$\psi(t)=e(\kappa\cdot t)$ for some  $\kappa\in\Z^l$, where $\cdot$
denotes the inner product operation.
  We refer to $\kappa$ as the
\emph{frequency} of $\psi$ and $\norm{\psi}=|\kappa|$ as the
\emph{frequency magnitude} of $\psi$.

If $p\colon\Z\to \R$ is a polynomial sequence of degree $k$,
then $p$ can be uniquely expressed in the form $
p(n)=\sum_{i=0}^k\binom{n}{i}\alpha_i $ where $\alpha_i\in\R$. For $N\in\N$ we
define
\begin{equation}\label{E:norms}
  \norm{e(p(n))}_{C^\infty[N]}=\max_{1\leq i\leq k}( N^i \norm{\alpha_i})
\end{equation}
where $\norm{t}=d(t,\Z)$.

Given $N\in\N$, a finite sequence $(g(n)\Gamma)_{1\leq n\leq N}$ is
said to be $\delta$-\emph{equidistributed} if
$$
\Big|\frac{1}{N}\sum_{n=1}^N F(g(n)\Gamma)-\int_{X}F \ dm_X\Big|\leq
\delta \norm{F}_{\text{Lip}(X)}
$$
for every Lipschitz function $F\colon X\to \C$ where
$
\norm{F}_{\text{Lip}(X)}=\norm{F}_\infty+ \sup_{x,y\in X, x\neq
  y}\frac{|F(x)-F(y)|}{d_X(x,y)}
$
for some appropriate metric $d_X$ on $X$.

The next result can be found  in \cite{FrW09} (see Theorem~3.4).
\begin{theorem}[{\bf Corollary
 of Green \&    Tao~\cite{GT09c}}]\label{T:GT2}
  Let $X=G/\Gamma$ be a connected nilmanifold (with $G_0$ simply connected), and $d\in\N$.

  Then for every small  enough $\delta>0$ there exists  $M=M_{X,d,\delta}\in \R$ with the following property: For every $N\in\N$,
  if $g\colon \Z\to G$ is a polynomial sequence of degree at most $d$
  such that the finite sequence $(g(n)\Gamma)_{1\leq n\leq N}$ is not
  $\delta$-equidistributed, then for some non-trivial quasi-character $\psi$ with $\norm{\psi}\leq M$ we have
  $$
    \norm{\psi(g(n))}_{C^\infty[N]}\leq  M
  $$
  where we think of $\psi$ as a
  character of some finite dimensional torus $\T^l$ (the affine
  torus) and $g(n)$ as a polynomial sequence of unipotent affine
  transformations on $\T^l$.
\end{theorem}
\begin{remark}
  We have $\psi(g(n))=e(p(n))$ for some $p\in \R[x]$ and therefore
  $\norm{\psi( g(n))}_{C^\infty[N]}$ is well defined.
\end{remark}

\subsection{Ergodic theory}\label{SS:BackgroundErgodic}
Below we gather some basic notions and facts from ergodic theory that
we use throughout the paper. The reader can find further background
material in ergodic theory in  \cite{Fu81}, \cite{Pe89},
\cite{Wa82}.
\subsubsection{Factors} A {\it homomorphism} from a system $(X,\X,\mu, T)$ onto a
system $(Y, \Y, \nu, S)$ is a measurable map $\pi\colon X'\to Y'$,
where $X'$ is a $T$-invariant subset of $X$ and $Y'$ is an
$S$-invariant subset of $Y$, both of full measure, such that
$\mu\circ\pi^{-1} = \nu$ and $S\circ\pi(x) = \pi\circ T(x)$ for $x\in
X'$. When we have such a homomorphism we say that the system $(Y, \Y,
\nu, S)$ is a {\it factor} of the system $(X,\X,\mu, T)$.  If
the factor map $\pi\colon X'\to Y'$ can be chosen to be injective,
then we say that the systems $(X,\X, \mu, T)$ and $(Y, \Y, \nu, S)$
are {\it isomorphic} (bijective maps on Lebesgue spaces have
measurable inverses).

A factor can be characterized (modulo isomorphism) by
$\pi^{-1}(\Y)$ which is a $T$-invariant sub-$\sigma$-algebra
of $\mathcal B$, and conversely any $T$-invariant sub-$\sigma$-algebra of
$\mathcal B$ defines a factor. By a classical abuse of terminology we
denote by the same letter the $\sigma$-algebra $\Y$ and its
inverse image by $\pi$. In other words, if $(Y, \Y, \nu, S)$ is a
factor of $(X,\X,\mu,T)$, we think of $\Y$ as a
sub-$\sigma$-algebra of $\X$. A factor can also be
characterized (modulo isomorphism) by a $T$-invariant subalgebra
$\mathcal{F}$ of $L^\infty(X,\X,\mu)$, in which case $\Y$ is
the sub-$\sigma$-algebra generated by $\mathcal{F}$, or equivalently,
$L^2(X,\Y,\mu)$ is the closure of $\mathcal{F}$ in
$L^2(X,\X,\mu)$. We will sometimes abuse notation and use the sub-$\sigma$-algebra $\mathcal{Y}$ in place of
 the subspace $L^2(X,\Y,\mu)$. For example, if we write that a function is orthogonal to the factor $\mathcal{Y}$,
  we mean  it  is orthogonal to the subspace $L^2(X,\Y,\mu)$.

If $\Y$ is a $T$-invariant sub-$\sigma$-algebra of $\X$ and $f\in
L^2(\mu)$, we define the {\it conditional expectation
  $\mathbb{E}(f|\Y)$ of $f$ with respect to $\Y$} to be the
orthogonal projection of $f$ onto $L^2(\Y)$. We will frequently make use
of the identities
$$
\int \mathbb{E}(f|\Y) \ d\mu= \int f\ d\mu, \quad
T\,\mathbb{E}(f|\Y)=\mathbb{E}(Tf|\Y).
$$
 If we want to indicate the dependence on the reference measure, we
 write $\mathbb{E}=\mathbb{E}_\mu$.




The transformation $T$ is {\it ergodic} if $Tf=f$ implies that $f=c$
(a.e.) for some $c\in \mathbb{C}$.
Every system $(X,\X,\mu,T)$ has an {\it ergodic decomposition},
meaning that we can write $\mu=\int \mu_t\ d\lambda(t)$, where
$\lambda$ is a probability measure on $[0,1]$ and $\mu_t$ are
$T$-invariant probability measures on $(X,\X)$ such that the systems
$(X,\X,\mu_t,T)$ are ergodic for $t\in [0,1]$.

We say that $(X,\X,\mu,T)$ is an {\it inverse limit of a sequence of
  factors} $(X,\X_j,\mu,T)$ if $(\X_j)_{j\in\mathbb{N}}$ is an
increasing sequence of $T$-invariant sub-$\sigma$-algebras such that
$\bigvee_{j\in\N}\X_j=\X$ up to sets of measure
zero.

\subsubsection{Seminorms and nilfactors}
Following \cite{HK05a},\footnote{In \cite{HK05a} the authors work
  with ergodic systems and real valued functions, but the whole discussion
  can be carried out for non-ergodic systems as well and complex
  valued functions without extra difficulties.} for every system $(X,\X,\mu,T)$ and
function $f\in L^\infty(\mu)$, we define inductively the
 seminorms $\nnorm{f}_\ell$ as follows: For $\ell=1$ we set
 $$
\nnorm{f}_1=\norm{\E(f|\mathcal{I})}_{L^2(\mu)}
$$
 where $\mathcal{I}$ is the
$\sigma$-algebra of $T$-invariant sets. For $\ell\geq 1$ we set
\begin{equation}\label{eq:recur}
  \nnorm f_{\ell+1}^{2^{\ell+1}} =\lim_{N\to\infty}
   \E_{1\leq n\leq N} \nnorm{\bar{f}\cdot T^nf}_{\ell}^{2^{\ell}}.\footnote{We remark that the limit is the same if the average  $\E_{1\leq n\leq N}$ is replaced with the average $\E_{n\in \Phi_N}$ where $(\Phi_N)_{N\in\N}$ is any F{\o}lner sequence in $\Z$.}
\end{equation}
It was shown in~\cite{HK05a} that for every integer
$\ell\geq 1$, $\nnorm\cdot_\ell$ is a seminorm on $L^\infty(\mu)$ and
it defines factors $\cZ_{\ell-1}=\cZ_{\ell-1}(T)$ in the following
manner: the $T$-invariant sub-$\sigma$-algebra $\cZ_{\ell-1}$ is
characterized by
$$
\text{ for } f\in L^\infty(\mu),\ \E(f|\cZ_{\ell-1})=0\text{ if and
  only if } \nnorm f_{\ell} = 0.
$$

It is shown in  \cite{HK05a} that for every $\ell\in \N$ the factor
 $\mathcal{Z}_\ell$
has a purely algebraic structure, in fact for all practical purposes
we can assume that it is an $\ell$-step nilsystem.
\begin{theorem}[{\bf Host \& Kra~\cite{HK05a}}]\label{T:HoKra}
  Let $(X,\X,\mu,T)$ be a system and $\ell\in \N$.

  Then  a.e. ergodic
   component of the  factor $\mathcal{Z}_\ell(T)$ is an inverse limit of $\ell$-step
  nilsystems.
\end{theorem}
Because of this result we call $\cZ_\ell$ the
\emph{$\ell$-step nilfactor} of the system. The smallest factor that is an extension of
all finite step nilfactors is denoted by $\cZ$ and is called the \emph{nilfactor} of
the system (in other words
$\cZ=\bigvee_{j\in\N}\cZ_j$.)
The nilfactor $\cZ$ is of particular interest because, as it turns out,
 it controls the  limiting behavior in $L^2(\mu)$ of
the  multiple
ergodic averages that are studied in Theorems~\ref{T:ConvSingle} and \ref{T:ConvSeveral}.

We also record two useful facts that can be easily established using
the  definition of the seminorms
\begin{equation}\label{E:seminonergodic}
\nnorm{f}_\ell^{2^\ell}=\int \nnorm{f}_{\mu_s,\ell}^{2^\ell} \ d\lambda(s), \quad
\quad   \nnorm{f\otimes\overline{f}}_{\mu\times \mu,\ell} \leq
\nnorm{f}_{\ell+1}^2
\end{equation}
where $\mu=\int \mu_s \ \! d\lambda(s)$ is the ergodic decomposition associated to the system $(X,\X,\mu,T)$.
 Hence,
if $T_t$ where $t\in [0,1]$ are the ergodic components of the transformation $T$,
then $\E(f|\mathcal{Z}_{\ell}(T))=0$ if and only if
$\E(f|\mathcal{Z}_\ell(T_t))=0$ for a.e. $t\in[0,1]$.
Also  if $f$  satisfies
$\E_\mu(f|\cZ_\ell(T))=0$, then $\E_{\mu\times\mu}(f\otimes
\overline{f}|\cZ_{\ell-1}(T\times T))=0$.

\section{Characteristic factors for multiples of a single sequence}\label{S:CharFactorsSingle}
A crucial step in the proof of Theorem~\ref{T:ConvSingle} is to show that for every $a\in\H$, not growing very fast or very slow, for every $\ell \in \N$, the nilfactor $\mathcal{Z}$ of a system is characteristic for the family $\{[a(n)],2[a(n)],\ldots, \ell [a(n)]\}$.
This is the context of Theorem~\ref{T:CharSingle}  which we are going to prove in this section.

 As is typically the case with such results, one assigns a notion of ``complexity'' to the relevant
 multiple ergodic   averages, and then uses induction on the ``complexity'' to prove the result.
This plan can be carried out  without serious  difficulties when the function $a\in \H$ satisfies
$t^{k-1}\log{t}\prec a(t)\prec t^{k}$ for some $k\in \N$. But when $a(t)=t\log{t}$, for example, there are serious difficulties caused by the fact that  the factor $\cZ$ is not characteristic for the Hardy sequence $[a'(n)]$, the reason being that the sequence $[\log{n}]$ grows too slowly. To deal with such functions we perform some initial maneuvers that enable us to transform the  problem to one where induction
on the ``complexity'' is applicable. Before giving the formal argument we informally explain how the initial step of the proof works in a model case.

\subsection{A model problem}
Suppose we want to  show that  the nilfactor $\cZ$ is characteristic for the family of sequences $\{[n\log{n}],2[n\log{n}],\ldots,\ell[n\log{n}]\}$.  Let
 $(X,\X,\mu,T)$ be a system, and suppose that
 one of the functions $f_1,\ldots,f_\ell\in L^\infty(\mu)$ is orthogonal to $\mathcal{Z}$. We have to show that
\begin{equation}\label{E:lwq}
\lim_{N\to \infty}\E_{1\leq  n\leq  N}V([n\log{n}])=0,
\end{equation}
where
$$
V(n)=T^nf_1\cdot T^{2n}f_2\cdot\ldots\cdot T^{\ell n}f_\ell,
$$
and the convergence takes place in $L^2(\mu)$. Our goal here is to show how to transform \eqref{E:lwq}
into a more manageable identity.

It will be more convenient for us to show that
\begin{equation}\label{E:initial}
\lim_{N\to \infty}\E_{N< n\leq  N+l(N)}V([n\log{n}])=0
\end{equation}
for some function $l(t)$ that satisfies $l(t)\prec t$
(Lemma~\ref{L:averaging} shows that \eqref{E:initial} implies \eqref{E:lwq}).
Using  the Taylor expansion of $a(t)=t\log{t}$ around the point $t=N$ we get for every $n\in\N$ that
$$
(N+n)\log(N+n)=N\log{N}+n(1+\log{N})+\frac{n^2}{2N}-\frac{n^3}{6\xi_n^2}
$$
for some $\xi_n\in [N,N+n]$.
Hence, if $c<2/3$,  then for every large  $N$   and $n=1,\ldots,[N^c]$,  we have
$$
[(N+n)\log(N+n)]=\Big[\alpha_N+n\beta_N +\frac{n^2}{2N}\Big]+e(n)
$$
for some $\alpha_N,\beta_N\in \R$ and error terms  $e(n)\in \{0,-1\}$.
Ignoring the error terms, and writing $[\sqrt{2N}]^2$ in place of $2N$ (all these
technical issues can be justified), we get that in order to establish  \eqref{E:initial} it suffices to show that
\begin{equation}\label{E:transformed}
\lim_{N\to \infty}\E_{1\leq  n\leq  N^c}V\Big(\Big[\alpha_N+n\beta_N+\frac{n^2}{[\sqrt{2N}]^2}\Big]\Big)=0.
\end{equation}
Since  every integer between $1$ and $N^c$ can be represented as
 $n[\sqrt{2N}]+r$ with $0\leq n\leq \tilde{l}(N)=N^{c}/[\sqrt{2N}]$  and
  $1\leq r\leq  r(N)=[\sqrt{2N}]$,
 and since
 $$\frac{(n [\sqrt{2N}] +r)^2}{[\sqrt{2N}]^2}+(n[\sqrt{2N}]+r)\alpha+\beta= n^2+n\alpha_{r,N}+\beta_{r,N}
 $$
 for some $\alpha_{r,N}, \beta_{r,N}\in \R$, we get that  \eqref{E:transformed}
follows if we show that
$$
\lim_{N\to \infty} \E_{1\leq r\leq r(N)} \Big(\E_{1\leq n\leq  \tilde{l}(N)}
V([\alpha_{r,N}+n\beta_{r,N}+n^2])\Big) =0.
$$
If we choose $c>1/2$, then $\tilde{l}(N)\to \infty$, and as a result the last identity
  follows if we show that
\begin{equation}\label{E:jkl''}
\lim_{N\to \infty} \sup_{\alpha,\beta\in \R}\norm{\E_{1\leq n\leq  N} V([\alpha+n\beta +n^2])}_{L^2(\mu)} =0.
\end{equation}
We have therefore  reduced matters to establishing uniform estimates for some   polynomial multiple ergodic averages,
and this turns out to be a more manageable problem.

We also remark that the argument used in the previous model example turns out to work for every function $a\in \H$
of polynomial growth that satisfies $|a(t)-p(t)|\succ \log{t} $ for every $p\in \R[t]$.
We give the details
in the next subsection.
\subsection{Proof of Theorem~\ref{T:CharSingle} modulo a polynomial ergodic theorem}
We are going to prove Theorem~\ref{T:CharSingle}  modulo the following
polynomial ergodic theorem that we shall prove in the next subsection:
\begin{proposition}\label{P:uniform}
Let $(X,\X,\mu, T)$ be a system, and suppose that at least one of the functions $f_1,f_2,\ldots, f_\ell\in L^\infty(\mu)$
is orthogonal to the nilfactor $\mathcal{Z}$.

Then for every $k\in \N$, nonzero $\alpha\in \R$, bounded two parameter sequence $(c_{N,n})_{N,n\in\N}$ of real numbers, and F{\o}lner sequence $(\Phi_N)_{N\in\N}$ in $\Z$ we have
$$
\lim_{N\to \infty} \sup_{p\in \R_{k-1}[t]}\norm{\E_{n\in \Phi_N} c_{N,n}\ \!T^{[n^k\alpha+p(n)]}f_1\cdot T^{2[n^k\alpha+p(n)]}f_2\cdot
 \ldots\cdot T^{\ell[n^k\alpha+p(n)]}f_{\ell}}_{L^2(\mu)}=0.
 $$
\end{proposition}
The main step in the deduction of Theorem~\ref{T:CharSingle} from Proposition~\ref{P:uniform} is carried out in Lemma~\ref{L:AwayPoly} below. Before delving into the proof of this lemma  we mention two  useful ingredients that
will be used in its proof. The first one was proved in \cite{Fr09}:
\begin{lemma}\label{L:cvb}
 Let $a\in\H$ have  polynomial growth and satisfy $|a(t)-p(t)|\succ \log{t}$  for every $p\in\R[t]$.

  Then for some $k\in \N$ we have
\begin{equation}\label{E:cvb}
 |a^{(k+1)}(t)| \text{ decreases to } 0, \quad 1/t^k\prec a^{(k)}(t)\prec 1,\
\text{ and } \   (a^{(k+1)}(t))^k\prec (a^{(k)}(t))^{k+1}.
\end{equation}
\end{lemma}
The second is the following simple result:
\begin{lemma}\label{L:averaging}
Let $(V(n))_{n\in\N}$ be a bounded sequence of vectors on a normed space.  Suppose that
$$
\lim_{N\to\infty}\big(\E_{N\leq n\leq N+l(N)}V(n)\big)=0
$$
for some positive function $l(t)$ with $l(t)\prec t$.
Then
$$
\lim_{N\to\infty} \E_{1\leq n\leq N} V(n)=0.
$$
\end{lemma}
\begin{proof}
We can cover the interval $[1,N]$ by a union  of non-overlapping intervals of the
 form $[k,k+l(k)]$, we denote this union  by $I_N$.
 Since $l(t)\prec t$ and the sequence
  $(V(n))_{n\in\N}$ is bounded we have that
$$
\lim_{N\to \infty} \E_{1\leq n\leq N}V(n)=\lim_{N\to \infty} \E_{n\in I_N}V(n).
$$
 Using our assumption, one easily gets that
 the limit $\lim_{N\to \infty} \E_{n\in I_N}V(n)$ is $0$, finishing the proof.
\end{proof}

\begin{lemma}\label{L:AwayPoly}
Let $(V(n))_{n\in\N}$ be a bounded sequence of vectors on a normed space.  Suppose that
for every $k\in\N$ and bounded two parameter sequence $(c_{N,n})_{N,n\in\N}$ of real numbers we have
$$
\lim_{N\to \infty}
\sup_{p\in \R_{k-1}[t],}
\norm{ \E_{1\leq  n\leq  N} c_{N,n}\ \!V(n^k+[p(n)])}=0.
$$

Then if   $a\in \H$ has polynomial growth and satisfies $|a(t)-p(t)|\succ \log{t}$ for every $p\in \R[t]$,
and $(c_n)_{n\in\N}$ is any bounded sequence of real numbers, we have
$$
\lim_{N\to\infty} \E_{1\leq n\leq N} c_n \ \!V([a(n)])=0.
$$
\end{lemma}
\begin{proof}
For convenience we assume that $c_n=1$ for every $n\in\N$, the proof is similar in the general case.
 By Lemma~\ref{L:averaging} it suffices to show that
\begin{equation}\label{E:uih}
\lim_{N\to \infty}\E_{N< n\leq  N+l(N)}V([a(n)])=0
\end{equation}
for some function $l(t)$ with $l(t)\prec t$ (we shall impose more conditions on $l(t)$ as the argument proceeds).

Let $k\in \N$ be such that the conclusion of Lemma~\ref{L:cvb} is satisfied, namely,
\begin{equation}\label{E:cvb1}
 |a^{(k+1)}(t)| \text{ decreases to } 0, \quad 1/t^k\prec a^{(k)}(t)\prec 1,\quad
\text{ and } \quad   (a^{(k+1)}(t))^k\prec (a^{(k)}(t))^{k+1}.
\end{equation}
For convenience we are going to assume that $a^{(k)}(t)$ is eventually positive, the proof is similar in the other case.

Using the Taylor expansion of $a(t)$ around the point $t=N$ we get for every $n\in\N$ that
\begin{equation}\label{E:Taylor1}
a(N+n)=a(N)+na'(N)+\cdots +\frac{n^k}{k!}a^{(k)}(N)+\frac{n^{k+1}}{(k+1)!}a^{(k+1)}(\xi_n)
\end{equation}
for some $\xi_n\in [N,N+n]$. Since $|a^{(k+1)}(t)|$ is  eventually decreasing, for every
large $N$ we have $|a^{(k+1)}(\xi_n)|\leq |a^{(k+1)}(N)|$.
It follows that if $l(t)$  satisfies
\begin{equation}\label{E:C1}
(l(t))^{k+1} a^{(k+1)}(t)\prec 1,
\end{equation}
 then for every large $N$ and  $n=1,\ldots, [l(N)]$ we have
\begin{equation}\label{E:C1'}
[a(N+n)]=
\Big[a(N)+na'(N)+\cdots +\frac{n^k}{k!}a^{(k)}(N)\Big]+e_N(n)
\end{equation}
where the error terms  $e_N(n)$ take values in the set $ \{0,-1\}$ (we used that
$a^{(k+1)}(t)$ is eventually negative).
For $t\in[0,1]$ and $x$ positive we have
$$
\left|\frac{1}{(x+t)^k}-\frac{1}{x^k}\right|\leq \frac{k}{x^{k+1}},
$$
therefore if
$$
d(t)=\frac{k!}{a^{(k)}(t)},
$$
then  setting  $x=[\sqrt[k]{d(N)}]$ and  $t=\{\sqrt[k]{d(N)}\}$ in the previous estimate
  we get
$$
\left|\frac{a^{(k)}(N)}{k!}-\frac{1}{[\sqrt[k]{d(N)}]^k}\right|\leq \frac{k}{[\sqrt[k]{d(N)}]^{k+1}}
\sim (a^{(k)}(N))^{1+\frac{1}{k}}.
$$
From this estimate and \eqref{E:C1'}, it follows that if
\begin{equation}\label{E:C2}
(l(t))^k  (a^{(k)}(t))^{1+\frac{1}{k}}\prec 1,
\end{equation}
then for every large $N$ and  $n=1,\ldots,[l(N)]$  we have
$$
[a(N+n)]=
\Big[a(N)+na'(N)+\cdots +\frac{n^k}{[\sqrt[k]{d(N)}]^k}\Big]+\tilde{e}_N(n)
$$
where the error terms $\tilde{e}_N(n)$ take values in the set $\{-2,-1,0,1\}$.
Hence, in order to establish \eqref{E:uih} it suffices to show that
\begin{equation}\label{E:dxc}
\lim_{N\to \infty}
\sup_{p\in \R_{k-1}[t]}
\norm{\E_{1\leq n\leq  l(N)}
V\Big(\Big[\frac{n^k}{[\sqrt[k]{d(N)}]^k}+p(n)\Big]+\tilde{e}_N(n)\Big)}=0.
\end{equation}

 Next notice that
 \eqref{E:dxc}  follows if we show that
  for every bounded sequence $(c_{N,n})_{N,n\in\N}$ we  have
\begin{equation}\label{E:nbm}
\lim_{N\to \infty}
\sup_{p\in \R_{k-1}[t]}
\norm{\E_{1\leq n\leq  l(N)}
c_{N,n}\ \! V\Big(\Big[\frac{n^k}{[\sqrt[k]{d(N)}]^k}+p(n)\Big]\Big)}=0.
\end{equation}
Indeed, it suffices to use \eqref{E:nbm} when  $c_{N,n}={\bf 1}_{\{k\colon \tilde{e}_N(k)=i\}}(n)$ for $i=-2,-1,0,1$, and then add the corresponding identities.

We perform one last maneuver by rewriting
 \eqref{E:nbm} in a more convenient form.
 Notice that  every integer between $1$ and $l(N)$ can be represented as
 $[\sqrt[k]{d(N)}]n+r$ with $0\leq n\leq \tilde{l}(N)=l(N)/[\sqrt[k]{d(N)}]$  and
  $1\leq r\leq  [\sqrt[k]{d(N)}]$.
Furthermore,  if we choose  $l(t)$ so that
\begin{equation}\label{E:C3}
 (l(t))^k a^{(k)}(t)\succ 1,
\end{equation}
then we have  $\tilde{l}(N)\to \infty$.
Since for every bounded sequence of vectors $V(n)$ the average $\E_{1\leq n\leq  b(N)}V(n)$ is
 equal to
 $\E_{1\leq r\leq  r(N)} \E_{1\leq n\leq  \tilde{l}(N)}V([\sqrt[k]{d(N)}]n+r)$
(up to negligible terms),  an easy computation (similar to the one used in the model example) shows that in order to
establish \eqref{E:nbm} it suffices to show that
  \begin{equation}\label{E:nbm'}
\lim_{N\to \infty}
\sup_{p\in \R_{k-1}[t],}
\norm{ \E_{1\leq  n\leq  N} c_{N,n}\ \! V([n^k+p(n)])}_{L^2(\mu)}=0.
\end{equation}

Summarizing, we have reduced matters to establishing \eqref{E:nbm'}, which holds by our assumption, provided that there exists a function $l(t)$ that
 satisfies  all the conditions imposed previously, namely,  $l(t)\prec t$ and  the conditions stated in  equations \eqref{E:C1}, \eqref{E:C2}, and  \eqref{E:C3}. Equivalently, the function $l(t)$ must satisfy
 $$
 \frac{1}{(a^{(k)}(t))^{\frac{1}{k}}}\prec l(t), \quad l(t)\prec t, \quad l(t) \prec   \frac{1}{(a^{(k+1)}(t))^{\frac{1}{k+1}}}, \quad \text{and}\quad   l(t)  \prec \frac{1}{(a^{(k)}(t))^{\frac{1}{k}+\frac{1}{k^2}}}.
 $$
That such a function  $l(t)$ exists follows from  the second and third conditions in \eqref{E:cvb1}, thus completing the proof.
\end{proof}
With the help of Proposition~\ref{P:uniform} and Lemma~\ref{L:AwayPoly} it is now easy to  prove  Theorem~\ref{T:CharSingle}.
\begin{proof}[Proof of Theorem~\ref{T:CharSingle}]
Let $(X,\X,\mu,T)$ be a system and suppose that at least one of the functions $f_1,\ldots,f_\ell\in L^\infty(\mu)$ is orthogonal to the nilfactor $\cZ$. Let $a\in \H$ have polynomial
growth and satisfy $a(t)\succ \log{t}$. We have to show that
\begin{equation}\label{E:zlo}
\lim_{N\to\infty}\E_{1\leq n\leq N} T^{[a(n)]}f_1\cdot\ldots\cdot T^{\ell[a(n)]}f_\ell=0
\end{equation}
where the convergence takes place in $L^2(\mu)$.

 Combining Proposition~\ref{P:uniform} and Lemma~\ref{L:AwayPoly}
we immediately get that \eqref{E:zlo} holds if $a\in \H$ has polynomial growth and satisfies
$|a(t)-p(t)|\succ \log{t}$ for every $p\in\R[t]$.
Therefore, it remains to deal with the case where
$a(t)=p(t)+e(t)$, where $p\in \R[t]$ is non-constant and $e(t)\ll \log{t}$.

Suppose first  that  $e(t)$ is bounded. Then $e(n)\to c$ for some $c\in \R$, and
as a result for every large $n\in\N$ we have
 $[a(n)]=[p(n)+c]+\tilde{e}(n)$ for some sequence $(\tilde{e}(n))_{n\in\N}$ with $\tilde{e}(n)\in \{0,\pm 1\}$.
Using this, we deduce that \eqref{E:zlo} follows
from Proposition~\ref{P:uniform}.

The last case to consider is when   $1\prec e(t)\prec \log{t}$.
 Let $I_m=\{n\in\N\colon [e(n)]=m\}$. Since $e(n+1)-e(n)\to 0$ (this follows from $e'(t)\to 0$ and the mean value theorem), and $e(n)\to \infty$,  it follows that for every large $m$ the set  $I_m$ is an integer interval with length that increases to infinity. Notice also that for  $n\in I_m$ we have   $[a(n)]=[p(n)]+m+\tilde{e}(n)$ for some sequence $(\tilde{e}(n))_{n\in\N}$ with $\tilde{e}(n)\in \{0,\pm 1\}$. Using this, we deduce that \eqref{E:zlo} follows
from Proposition~\ref{P:uniform}. This completes the proof.
\end{proof}

\subsection{Proof of the polynomial ergodic theorem}\label{SS:PET1}
Let $\mathcal{P}=\{p_1,\ldots,p_\ell\}$  be a family of polynomials with real coefficients.
We say that the family $\mathcal{P}$ consists of \emph{non-constant  and essentially
  distinct} polynomials, if  all the polynomials and their  pairwise differences have positive degree.
The maximum degree of the polynomials  is called the \emph{degree} of the polynomial family, and is denoted by  $\text{deg}(\mathcal{P})$.
 Given a
polynomial family $\mathcal{P}$, let $\mathcal{P}_i$ be the subfamily
of polynomials of degree $i$ in $\mathcal{P}$. We let $w_i$ denote the
number of distinct leading coefficients that appear in the family
$\mathcal{P}_i$. The vector $(d,w_d,\ldots,w_1)$, where $d=\text{deg}(\mathcal{P})$, is called the
\emph{type} of the polynomial family $\mathcal{P}$.

We will use an induction scheme, often called PET induction
(Polynomial Exhaustion Technique), on types of polynomial families
that was introduced by Bergelson in \cite{Be87a}.  We order
the set of all possible types lexicographically, meaning,
$(d,w_d,\ldots,w_1)>(d',w_d',\ldots,w_1')$ if and only if in the first
instance where the two vectors disagree the coordinate of the first
vector is greater than the coordinate of the second vector.

Given a  family of non-constant essentially distinct polynomials $\mathcal{P}=\{p_1,\ldots p_\ell\}$, a positive integer $h$,
and $p\in \mathcal{P}$, we form a new family $\mathcal{P}(p,h)$ as follows:
We start with the family of polynomials
$$
\{ p_1(t+h)-p(t),\ldots,p_\ell(t+h)-p(t), p_1(t)-p(t),\ldots, p_\ell(t)-p(t)\},
$$
and successively remove the smallest number of polynomials so that the resulting family consists
of  non-constant, essentially distinct polynomials. Then for every large $h$ the function
$p_i(t+h)-p(t)$ will be  removed  if and only if    $p_i$ is linear (then $(p_i(t+h)-p(t))-(p_i(t)-p(t))=p_i(h)$),  and the function $p_i(t)-p(t)$ will be removed if and only if $p=p_i$.
\begin{example}
If  $\mathcal{P}=\{t,2t, t^{2}\}$ and $p(t)=t$, then
we start with the family of polynomials
$$\{h , t+2h, (t+h)^{2}- t, 0, t , t^2 -t\}
$$
and remove the first, second, and  fourth polynomials to get
$$
\mathcal{P}(t,h)=\{ (t+h)^{2}- t,t, t^2 -t\}.
$$
Notice that the family $\mathcal{P}$ has type $(2,1,2)$, and the family
$\mathcal{P}(t,h)$ has smaller type, namely, $(2,1,1)$.
\end{example}

The main step in the proof of
Proposition~\ref{P:uniform}  is carried out in  Lemma~\ref{L:uniform}. This lemma
is proved using  induction on the type of the family of functions
involved. In order to carry out the inductive step we will use the following:
\begin{lemma}\label{L:indstep2}
Let $\mathcal{P}=\{p_1,\ldots p_\ell\}$ be  family of non-constant essentially distinct polynomials,
 and suppose that  $\text{deg}(p_1)=\text{deg}(\mathcal{P})\geq 2$.

Then there exists
$p\in \mathcal{P}$ such that for every large $h$ the family $\mathcal{P}(p,h)$
has type smaller than that of $\mathcal{P}$, and
$\text{deg}(p_1(t+h)-p(t))=\text{deg}(\mathcal{P}(p,h))$.
\end{lemma}
\begin{remark} Since $\text{deg}(p_1)\geq 2$,  no-matter what the choice of $p$ will be, the polynomial
$p_1(t+h)-p(t)$ is going to be an element of the family $\mathcal{P}(p,h)$ for every large $h$.
\end{remark}
\begin{proof}
Suppose first that $\text{deg}(p_i)<\text{deg}(p_1)$ for some $i\in\{2,\ldots,\ell\}$. Let $i_0$ be such that the polynomial $p_{i_0}$ has  minimal degree. Then $p=p_{i_0}$ has the advertised property.

Otherwise,  all the polynomials have the same degree, in which case for $i=2,\ldots,\ell$ we have  $p_i(t)=\alpha_ip_1(t)+q_i(t)$ for some
  non-zero real numbers $\alpha_2,\ldots, \alpha_\ell$ and polynomials $q_i$ with
  $\text{deg}(q_i)<\text{deg}(p_1)$.
If $\alpha_{i_0}\neq 1$ for some $i_0\in \{2,\ldots,\ell\}$, then  $p=p_{i_0}$ has the advertised property. If $\alpha_i= 1$ for $i=2,\ldots,\ell$, let
 $i_0$ be such that the function $q_{i_0}$ has maximal degree.  Then  $p=p_{i_0}$ has the advertised property. This completes the proof.
\end{proof}
We are also going to use   a variation of the classical elementary lemma of van der Corput.
Its proof is a straightforward modification of  the one given in \cite{Be87a}, therefore we omit it.
\begin{lemma}\label{L:N-VDC}
Let  $\{v_{N,n}\}_{N,n\in\N}$ be a bounded  sequence of vectors in a
Hilbert space, and $(\Phi_N)_{N\in\N}$ be a F{\o}lner sequence of subsets of $\N$.
 For every $h\in \N$ we set
$$
b_h=\overline{\lim}_{N\to\infty}\Big|
\E_{n\in\Phi_N}<v_{N,n+h},v_{N,n}>\Big|.
$$

Then
$$
\overline{\lim}_{N\to\infty}
\norm{\E_{n\in\Phi_N}v_{N,n}}^2\leq 4 \ \!
\overline{\lim}_{H\to\infty}\E_{1\leq h\leq H} b_h.
$$
\end{lemma}
 To state our next result it will be convenient to introduce some notation. For $N\in\N$ let $\mathcal{P}_N=\{p_{1,N},\ldots,p_{\ell,N}\}$ be a family of polynomials with real coefficients.
We say that the collection  $(\mathcal{P}_N)_{N\in\N}$  is \emph{``nice''} if  for every $N\in \N$  the polynomials $p_{i,N}$ and $p_{i,N}-p_{j,N}$ (for $i\neq j$) are non-constant and their leading coefficients are independent of $N$.
\begin{lemma}\label{L:uniform}
Let $(\{p_{1,N},\ldots,p_{\ell,N}\})_{N\in\N}$  be
a ``nice'' collection of polynomial families.
Let $(X,\X,\mu,T)$ be a system, and
  suppose that one of the functions   $f_1,\ldots,f_\ell\in L^\infty(\mu)$ is orthogonal to the nilfactor $\cZ$.

Then for every F{\o}lner sequence $(\Phi_N)_{N\in \N}$ in $\Z$ and bounded sequence $(c_{N,n})_{N,n\in\N}$ we have
\begin{equation}\label{E:mnm}
\lim_{N\to \infty}\E_{n\in \Phi_N} c_{N,n} \ \!T^{[p_{1,N}(n)]}f_1\cdot \ldots\cdot T^{[p_{\ell,N}(n)]}f_\ell=0
\end{equation}
where the convergence  takes place in $L^2(\mu)$.
\end{lemma}
\begin{remark}
In the special case where $p_{i,N}=p_i$ for $i=1,\ldots,\ell$ we get a different proof\footnote{In contrast with the proof given  in \cite{Lei05c}, we do  not have to  work with polynomials of several variables (which was a key trick in \cite{Lei05c}) in order to prove the single variable result.} of the known result
that the nilfactor $\cZ$ is characteristic for any family of non-constant,  essentially distinct polynomials of a single variable.
\end{remark}
\begin{proof}
Without loss of generality  we can assume that the function $f_1$ is orthogonal to the nilfactor $\cZ$. Furthermore, we can assume that $\norm{f_i}_\infty \leq 1$ for $i=1,\ldots,\ell$, and $|c_{N,n}|\leq 1$ for every $N,n\in\N$.
It will be crucial for our argument to assume that the polynomial $p_{1,N}$ has maximal degree within the family $\mathcal{P}_N=\{p_{1,N},\ldots,p_{\ell,N}\}$. To get this extra assumption it is convenient to somewhat modify our goal; instead of \eqref{E:mnm} we shall prove that
\begin{equation}\label{E:mnm'}
\lim_{N\to\infty}\sup_{\norm{f_0}_\infty, \norm{f_2}_{\infty}
,\ldots,\norm{f_\ell}_\infty\leq 1} \E_{n\in \Phi_N} \Big|\int
f_0\cdot  T^{[p_{1,N}(n)]}f_1\cdot T^{[p_{2,N}(n)]}f_2\cdot\ldots \cdot
T^{[p_{\ell,N}(n)]}f_\ell \ d\mu\Big| =0.
\end{equation}

Notice first that \eqref{E:mnm'}  implies \eqref{E:mnm}. Indeed, \eqref{E:mnm'}
gives that
\begin{equation}\label{E:mnm''}
\lim_{N\to\infty}
\E_{n\in \Phi_N} c_{N,n} \ \! \int
f_{0,N}\cdot  T^{[p_{1,N}(n)]}f_1\cdot T^{[p_{2,N}(n)]}f_2\cdot\ldots \cdot
T^{[p_{\ell,N}(n)]}f_\ell \ d\mu =0
\end{equation}
whenever $f_{0,N}\in L^\infty(\mu)$ satisfies $\norm{f_{0,N}}_\infty \leq 1$ for $N\in\N$.
Using \eqref{E:mnm''} with the conjugate of the function  $\E_{n\in \Phi_N} c_{N,n} \ T^{[p_{1,N}(n)]}f_1\cdot T^{[p_{2,N}(n)]}f_2\cdot\ldots \cdot
T^{[p_{\ell,N}(n)]}f_\ell$ in place of the function $f_{0,N}$ (for every $N\in \N$), we get \eqref{E:mnm}.

 Next we claim that when proving \eqref{E:mnm'} we can assume that the polynomial $p_{1,N}$ has maximal degree within the family  $\mathcal{P}_N$. Indeed, if this is not the case, then $\text{deg}(p_{1,N})<\text{deg}(p_{i,N})$ for
some $i=2,\ldots,\ell$, say this happens for $i=\ell$. After factoring out the transformation $T^{[p_{\ell,N}(n)]}$  we see  that \eqref{E:mnm'} can be rewritten as
\begin{multline}\label{E:mnm1}
\lim_{N\to\infty}\sup_{\norm{f_0}_\infty, \norm{f_2}_{\infty}
,\ldots,\norm{f_\ell}_\infty\leq 1} \E_{n\in \Phi_N} \Big|\int
f_\ell\cdot T^{[-p_{\ell,N}(n)]+e_{0,N}(n)}f_0\cdot  T^{[p_{1,N}(n)-p_{\ell,N}(n)]+e_{1,N}(n)}f_1\cdot \\
T^{[p_{2,N}(n)-p_{\ell,N}(n)]+e_{2,N}(n)}f_2\cdot\ldots \cdot
T^{[p_{\ell-1,N}(n)-p_{\ell,N}(n)]+e_{\ell-1,N}(n)}f_{\ell-1} \ d\mu\Big| =0
\end{multline}
for some error terms $e_{i,N}(n)$ with values in the set $\{0,1\}$.
Furthermore,  notice that \eqref{E:mnm1} follows if we show that
\begin{multline}\label{E:mnm2}
\lim_{N\to\infty}\sup_{\norm{f_0}_\infty, \norm{f_2}_{\infty}
,\ldots,\norm{f_\ell}_\infty\leq 1} \E_{n\in \Phi_N} \Big|\int
f_\ell\cdot T^{[-p_{\ell,N}(n)]}f_0\cdot  T^{[p_{1,N}(n)-p_{\ell,N}(n)]}f_1\cdot \\
T^{[p_{2,N}(n)-p_{\ell,N}(n)]}f_2\cdot\ldots \cdot
T^{[p_{\ell-1,N}(n)-p_{\ell,N}(n)]}f_{\ell-1} \ d\mu\Big| =0.
\end{multline}
Since the collection of polynomial  families $(\mathcal{P}'_N)_{N\in\N}$, where
$$
\mathcal{P}'_N=\{-p_{\ell,N}, p_{1,N}-p_{\ell,N},\ldots,p_{\ell-1,N}-p_{\ell,N}\},
$$
 is also ``nice'', and the polynomial
$p_{1,N}-p_{\ell,N}$ has maximal degree within the family $\mathcal{P}_N'$, the claim follows.

Summarizing,
we have reduced matters to establishing \eqref{E:mnm'} for every system, assuming that the function $f_1\in L^\infty(\mu)$
is orthogonal to the nilfactor $\cZ$ and the polynomial $p_{1,N}$ has maximal degree
within the family $\mathcal{P}_N$.
We shall do  this by using induction on the type of the family  of  polynomials $\mathcal{P}_N$ (the type of this family is independent of $N$).

The case where all the polynomials have degree $1$ can be  treated as in the proof of  \eqref{E:ofi1} in Proposition~\ref{P:CharSublin} below. (For linear functions, the same argument works for any F{\o}lner sequence
$\Phi_N$ in place of the intervals $[1,N]$; we leave the  routine details to the reader.)

 Now let $d\geq 2$ and suppose that the statement  holds for every ``nice''collection of polynomial families with
  type smaller than $(d,w_d,\ldots,w_1)$. Let $(\mathcal{P}_N)_{N\in\N}$, where  $\mathcal{P}_N=\{p_{1,N},\ldots,p_{\ell,N}\}$, be a ``nice'' collection of polynomial families
   with type $(d,w_d,\ldots,w_1)$.

  Using  the Cauchy-Schwarz inequality we see that
\eqref{E:mnm'} follows if   we show that
$$
\lim_{N\to\infty}\sup_{\norm{f_0}_\infty, \norm{f_2}_{\infty}
,\ldots,\norm{f_\ell}_\infty\leq 1}
\E_{n\in \Phi_N}\Big| \int f_{0}\cdot  T^{[p_{1,N}(n)]}f_1\cdot T^{[p_{2,N}(n)]}f_{2}\cdot\ldots \cdot T^{[p_{\ell,N}(n)]}f_{\ell} \ d\mu\Big|^2=0.
$$
This last identity can be rewritten as
 $$
\lim_{N\to\infty}\sup_{\norm{f_0}_\infty, \norm{f_2}_{\infty}
,\ldots,\norm{f_\ell}_\infty\leq 1}
\E_{n\in \Phi_N} \int F_{0}\cdot  S^{[p_{1,N}(n)]}F_1\cdot S^{[p_{2,N}(n)]}F_{2}\cdot \ldots \cdot S^{[p_{\ell,N}(n)]}F_{\ell} \ d(\mu\times \mu)=0
$$
where $S=T\times T$ and  $F_{i}=f_{i} \otimes \overline{f}_{i}$ for  $i=0,1,\ldots,\ell.$
Using the
 Cauchy-Schwarz inequality we see that it suffices to show that
 \begin{equation}\label{E:oro}
\lim_{N\to\infty}\sup_{\norm{F_2}_{\infty}
,\ldots,\norm{F_\ell}_\infty\leq 1}
\norm{\E_{n\in \Phi_N} \cdot  S^{[p_{1,N}(n)]}F_1\cdot S^{[p_{2,N}(n)]}F_{2}\cdot \ldots \cdot S^{[p_{\ell,N}(n)]}F_{\ell}}_{L^2(\mu\times \mu)}=0
\end{equation}
where $F_1=f_1 \otimes \overline{f}_1$.
We choose  functions $F_{i,N}$, $i=2,\ldots,\ell$,  with  sup norm at most $1$, so that the value of the norms in \eqref{E:oro} is $1/N$ close to the supremum and use
  Lemma~\ref{L:N-VDC}. We get that \eqref{E:oro} follows if we show that
for every large  $h$ we have
\begin{multline*}
\lim_{N\to\infty}\sup_{ \norm{F_2}_{\infty}
,\ldots,\norm{F_\ell}_\infty\leq 1}
\Big| \E_{ n\in \Phi_N} \int   S^{[p_{1,N}(n+h)]}F_1\cdot S^{[p_{2,N}(n+h)]}F_{2}\cdot \ldots \cdot S^{[p_{\ell,N}(n+h)]}F_{\ell} \cdot\\
S^{[p_{1,N}(n)]}\overline{F}_1\cdot S^{[p_{2,N}(n)]}\overline{F}_{2}\cdot \ldots \cdot S^{[p_{\ell,N}(n)]}\overline{F}_{\ell}
\ d(\mu\times \mu)\Big|=0.
\end{multline*}
Factoring out the transformation $S^{[p_N(n)]}$,
 where $p_N=p_{i,N}$ for some $i\in \{1,\ldots,\ell\}$ is chosen as  in Lemma~\ref{L:indstep2} (the choice of $i$ is independent of $N$), we can rewrite the last identity as
  \begin{multline*}
\lim_{N\to\infty}\sup_{\norm{F_2}_{\infty}
,\ldots,\norm{F_\ell}_\infty\leq 1}
\E_{n\in \Phi_N} \Big|\int   S^{[p_{1,N}(n+h)-p_N(n)]+e_{1,N}(h,n)}F_1\cdot \ldots \cdot S^{[p_{\ell,N}(n+h)-p_N(n)]+e_{\ell,N}(h,n)}F_{\ell} \cdot\\
S^{[p_{1,N}(n)-p_N(n)]+e_{\ell+1,N}(h,n)}\overline{F}_1\cdot \ldots \cdot S^{[p_{\ell,N}(n)-p_N(n)]+e_{2\ell,N}(h,n)}\overline{F}_{\ell}
\ d(\mu\times \mu)\Big|=0
\end{multline*}
for some error terms  $e_{i,N}(h,n) $ with values in $\{0,1\}$.
For every fixed $h,N\in \N$ we can partition the integers into a finite number of sets,
  that depend only on $\ell$,  where all sequences $e_{i,N}(h,n)$ are constant.
Therefore, the last identity
  follows if we show that for every large  $h$ we have
  \begin{multline}\label{E:aris1}
\lim_{N\to\infty}\sup_{\norm{F_2}_{\infty}
,\ldots,\norm{F_{\ell}}_\infty\leq 1}
\E_{n\in \Phi_N} \Big|\int   S^{[p_{1,N}(n+h)-p_N(n)]}F_1\cdot \ldots \cdot S^{[p_{\ell,N}(n+h)-p_N(n)]}F_\ell \cdot\\
S^{[p_{1,N}(n)-p_N(n)]}F_1\cdot \ldots \cdot S^{[p_{\ell,N}(n)-p_N(n)]}\overline{F}_\ell
\ d(\mu\times \mu)\Big|=0.
\end{multline}
Next notice that if  the polynomial $p_{i,N}$ have degree $1$  (this can only happen for $i\neq 1$), then $p_{i,N}(n+h)=p_{i,N}(n)+c_{N}(h)$ for some $c_{N}(h)\in \R$. Hence, for those values of $i$
we can write
$$
S^{[p_{i,N}(n+h)-p_N(n)]}F_i\cdot S^{[p_{i,N}(n)-p_N(n)]}\overline{F}_i=S^{[p_{i,N}(n)-p_N(n)]}(S^{[c_{N}(h)]+e_{N}(h,n)}F_i\cdot \overline{F}_i),
$$
for some error terms $e_N(h,n)\in \{0,1\}$. As explained before, because the error terms $e_{N}(h,n)$ take values in a finite set, when proving \eqref{E:aris1} we  can assume that $e_{N}(h,n)=0$.
Therefore,  we have reduced matters to showing  that for every large $h$ we have
  \begin{multline}\label{E:aris1'}
\lim_{N\to\infty}\sup_{\norm{F_0},\norm{F_2}_{\infty}
,\ldots,\norm{F_{r}}_\infty\leq 1}\Big|\E_{n\in \Phi_N}\int F_0\cdot
    S^{[p_{1,N}(n+h)-p_N(n)]}F_1\cdot
 S^{[q_{2,N}(n)]}F_2 \cdot
 \ldots \cdot\\ S^{[q_{r,N}(n)]}F_r \ d(\mu\times\mu)\Big|=0.
\end{multline}
 for some $r\in \N$, where  all the polynomials  involved are elements of the collection of
 polynomial families $(\mathcal{P}_N(p_N,h))_{N\in\N}$ (defined on the beginning of this subsection).
 For every large $h$
   this new collection of polynomial families is ``nice'' and
  by Lemma~\ref{L:indstep2}  it has type smaller than that of
 $(\mathcal{P}_N)_{N\in\N}$. Furthermore,
  the degree of the polynomial $p_{1,N}(n+h)-p_N(n)$
  is maximal within the family $\mathcal{P}_N(p_N,h)$. Since $f_1$ is orthogonal to the factor $\cZ(T)$, we have that $F_1=f_1\otimes\overline{f}_1$ is orthogonal to the factor $\cZ(S)$.
 Therefore, the induction hypothesis applies, and verifies \eqref{E:aris1'}  for every large $h$. This completes the induction and the proof.
 \end{proof}
 We are now ready to give the proof of Proposition~\ref{P:uniform}.
\begin{proof}[Proof of Proposition~\ref{P:uniform}]
Let $(X,\X,\mu,T)$ be a system, and suppose that at least one
 of the functions $f_1,\ldots,f_\ell\in L^\infty(\mu)$ is orthogonal to the nilfactor $\cZ$.
Our goal is to show that for every bounded two variable sequence $(c_{N,n})_{n\in\N}$ we have
\begin{equation}\label{E:aba}
\lim_{N\to\infty}\sup_{p\in \R_{k-1}[t]}\norm{\E_{n\in \Phi_N}c_{N,n} \ \!
 T^{[n^k\alpha +p(n)]}f_1\cdot  T^{2[n^k\alpha+p(n)]}f_2\cdot\ldots \cdot
T^{\ell[n^k\alpha+p(n)]}f_\ell}_{L^2(\mu)} =0.
\end{equation}
We choose  polynomials $p_{N}\in \R_{k-1}[t]$, so that the norm in \eqref{E:aba} is $1/N$ close to the supremum. Then \eqref{E:aba} takes the form
\begin{equation}\label{E:aba'}
\lim_{N\to\infty}\norm{\E_{n\in \Phi_N}c_{N,n} \ \!
 T^{[n^k\alpha +p_N(n)]}f_1\cdot  T^{2[n^k\alpha+p_N(n)]}f_2\cdot\ldots \cdot
T^{\ell[n^k\alpha+p_N(n)]}f_\ell}_{L^2(\mu)} =0.
\end{equation}
Equation \eqref{E:aba'} can be rewritten as
\begin{equation}\label{E:aba''}
\lim_{N\to\infty}\E_{n\in \Phi_N}c_{N,n} \ \!
 T^{[n^k\alpha +p_N(n)]}f_1\cdot  T^{[2(n^k\alpha+p_N(n))]+e_{2,N}(n)}f_2\cdot\ldots \cdot
T^{[\ell(n^k\alpha+p_N(n))]+e_{\ell,N}(n)}f_\ell =0
\end{equation}
where convergence takes place in $L^2(\mu)$ and  the error terms $e_{2,N}(n),\ldots, e_{\ell,N}(n)$ take
 values in the set $\{0,-1,\ldots,-\ell\}$.   Using  Lemma~\ref{L:uniform} we deduce that
 \eqref{E:aba''} holds, completing the proof.
\end{proof}

\section{Characteristic factors for several sequences}\label{S:CharFactorSeveral}
A crucial step in the proof of Theorem~\ref{T:ConvSeveral} is to show that the nilfactor $\mathcal{Z}$ is characteristic for the related multiple ergodic averages. This is the context of Theorem~\ref{T:CharSeveral}
 which we are going to prove in this section.
Our proof  extends  an argument used in \cite{BH09} where  a similar result was verified for weakly mixing systems.
 Since there are a few non-trivial extra complications in our case,
we shall give our proof in more or less full detail, referring the reader to \cite{BH09}
only when an argument we  need is a straightforward modification of one used there.
\subsection{The sub-linear case.}
First, we are going to  prove Theorem~\ref{T:CharSeveral} in the case where all the  functions have at most  linear growth.

   We first give two  lemmas that will be used in our proof.  A variation of the first one appears already  in \cite{BH09}:
 \begin{lemma}\label{L:ChangeVar}
   Let $(V_M(n))_{M,n\in\N}$ be a bounded two parameter sequence of vectors in a normed space and suppose that $a\in \H$
   satisfies $t^{\varepsilon}\prec a(t)\prec t$ for some $\varepsilon>0$.
   Then
   there exists a positive constant $C_a$ such that
  $$ \limsup_{N\to \infty} \sup_{M\in \N}\norm{\E_{1\leq n\leq N} V_M([a(n)])}\leq C_a \limsup_{N\to \infty}
   \sup_{M\in\N}\norm{\E_{1\leq n\leq N} V_M(n)}.
   $$
   \end{lemma}
   \begin{proof}
Without loss of generality we assume that $a(t)$ is positive. Let
$w(n)=|m\in \N\colon [a(m)]=n|$ and $W(n)=\sum_{m=1}^n w(m)$. Using
our hypothesis on $a(t)$ it is not hard  to verify (for the details
see Lemma 2.5 in \cite{BH09}) that
\begin{equation}\label{E:basic10}
w(n) \text{ is eventually increasing},  \quad \text{and} \quad
\frac{nw(n)}{W(n)} \text{ is bounded}.
\end{equation}
Using the definition of $w(n)$ and $W(n)$, that $w(n)/W(n)\to 0$,
and that $(V_M(n))_{M,n\in\N}$ is bounded, we conclude that
   $$
   \lim_{N\to\infty}\sup_{M\in \N}\norm{\E_{1\leq n\leq N} V_M([a(n)])-
    \frac{1}{W([a(N)])}\sum_{n=1}^{[a(N)]} w(n) V_M(n)}=0.
   $$
   So in order to get the asserted estimate it suffices to estimate
   the expression
   \begin{equation}\label{E:ToEstimate}
\lim_{N\to\infty}\sup_{M\in \N}\norm{ \frac{1}{W(N)}\sum_{n=1}^{N}
w(n) V_M(n)}.
   \end{equation}
Let
$$
A_M(n)=\E_{1\leq k\leq n} V_M(k).
$$
 Using summation by parts we have
   $$
   \frac{1}{W(N)}\sum_{n=1}^N w(n) V_M(n)=
   \frac{1}{W(N)}\Big( Nw(N)A_M(N)+\sum_{n=2}^{N}n (w(n-1)-w(n))A_M(n) \Big).
   $$
We take norms, then the sup over $M$, and then the limsup as
$N\to\infty$. Letting
$$
L:=\limsup_{n\to\infty} \sup_{M\in \N} \norm{A_M(n)},
$$
we get  that the quantity in \eqref{E:ToEstimate} is bounded by
  $$
 L\cdot \Big( \limsup_{N\to\infty} \frac{Nw(N)}{W(N)} +
    \limsup_{N\to\infty} \frac{1}{W(N)}\sum_{n=2}^{N}n (w(n)-w(n-1))\Big)
$$
where we used that $W(n)\to +\infty$  and that  $w(n)$ is eventually
increasing (see \eqref{E:basic10}).
  The asserted estimate follows upon noticing that   $Nw(N)/W(N)$ is
  bounded (see \eqref{E:basic10}) and also
 $$
  \frac{N \sum_{n=2}^N (w(n)-w(n-1))}{W(N)}=
 \frac{N (w(N)-w(1))}{W(N)}
   $$
is bounded. Note also that all implicit constants depend only on the
function $a(t)$. The proof is complete.
   \end{proof}
The second lemma is a well known estimate, we prove it for
completeness.
\begin{lemma}\label{L:ccx}
Let $(X,\X,\mu, T)$ be a system  and $f_1\in L^\infty(\mu)$.

Then  we have
\begin{equation}\label{E:rtq}
\limsup_{N\to\infty}\sup_{\norm{f_0}_\infty\leq 1}
\E_{1\leq n\leq N} \Big|\int f_0\cdot  T^{n}f_1 \ d\mu\Big|\leq  \nnorm{f_1}_2.
\end{equation}
\end{lemma}
\begin{proof}
Using the Cauchy-Schwarz inequality we see that
$$
\Big(\E_{1\leq n\leq N} \Big|\int f_0\cdot  T^{n}f_1 \ d\mu\Big|\Big)^2\leq
\E_{1\leq n\leq N} \Big|\int f_0\cdot  T^{n}f_1 \ d\mu\Big|^2.
$$
The last average can be rewritten as
$$
\E_{1\leq n\leq N} \int F_0\cdot  S^{n}F_1 \ d(\mu\times\mu)
$$
where $S=T\times T$, $F_0=f_0\otimes \overline{f}_0$, and $F_1=f_1\otimes \overline{f}_1$.
Assuming that $\norm{f_0}_{\infty}\leq 1$, and using the Cauchy-Schwarz inequality again,
we find that the last average is bounded by
$$
\norm{\E_{1\leq n\leq N} S^{n}F_1}_{L^2(\mu\times\mu)}.
$$
Taking limits, and using the ergodic theorem,  we conclude that the square of the limit in \eqref{E:rtq}
is bounded by
$$
 \norm{\E_{\mu\times\mu}(f_1\otimes\overline{f}_1|\mathcal{I}(S))}_{L^2(\mu\times\mu)}\leq\nnorm{f_1}^2_2
$$
where the last estimate follows   from  \eqref{E:seminonergodic} and
the definition of $\nnorm{f}_1$. This establishes the advertised
estimate.
\end{proof}
In the proof of Proposition~\ref{P:CharSublin} we are going to use the symbol $\ll_{w_1,\ldots, w_k}$  when
 some expression is  majorized by some  other expression and  the implied constant
 depends on the parameters $w_1,\ldots, w_k$.
\begin{proposition}\label{P:CharSublin}
Suppose that $a_1,\ldots,a_\ell\in \mathcal{LE}$ satisfy $t^{\varepsilon}\prec a_i(t)\ll t$
for some $\varepsilon>0$, and
 the same is true for the functions $a_i(t)-a_j(t)$ for $i\neq j$.

 Then the factor $\cZ$ is characteristic for the scheme
 $\{[a_1(n)], \ldots, [a_\ell(n)]\}$.
\end{proposition}
\begin{proof}
Let us first remark that in order to carry out our argument it will be convenient to work with a Hardy field $H$
that contains $\LE$ and such that the set $H^+=\{a\in H\colon a(t)\to \infty\}$  is closed under composition and compositional inversion
($\LE$ does not have this last property).
Such a Hardy field exists, in fact as mentioned in \cite{Bos94}, a constructive example is the field of germs at $\infty$ of Pfaffian functions, which we denote by $\mathcal{P}$.\footnote{This is defined inductively as follows:
We let $\mathcal{P}_1$  be the set of all    $f\in C^\infty(\R_+)$  that satisfy $f'=p(t,f)$  for some $p\in \Z[t_0,t_1]$, and for $k\geq 2$ we let  $\mathcal{P}_{k}$  be the set of all  $f\in C^\infty(\R_+)$  that satisfy $f'=p(t,f_1,\ldots,f_{k-1},f)$  for some $p\in \Z[t_0,t_1,\ldots,t_{k}]$ and $f_i\in \mathcal{P}_i$.
Then $\mathcal{P}$ is the set of germs of functions in $\bigcup_{k\in\N}\mathcal{P}_k$. }
(We are not going
to make use of the exact structure of $\mathcal{P}$.)
Our goal is to show that if the functions $a_1,\ldots,a_\ell\in \mathcal{P}$ satisfy
 the stated assumptions, and  one of the functions $f_1,\ldots,f_\ell$ is orthogonal to the nilfactor $\cZ$, then
\begin{equation}\label{E:paok1}
\lim_{N\to\infty}\norm{\E_{1\leq n\leq N}  T^{[a_1(n)]}f_1\cdot
T^{[a_2(n)]}f_2\cdot\ldots \cdot T^{[a_\ell(n)]}f_{\ell}}_{L^2(\mu)}
=0.
\end{equation}
Without loss of generality we can assume that the function $f_1$ is orthogonal to  $\cZ$.

We are going to use induction on $\ell$ to  show the following:
If $(X,\X,\mu , T)$ is a system,  $f_1$
satisfies  $\norm{f_1}_\infty\leq 1$, and the functions $a_1,\ldots,a_\ell\in \mathcal{P}$ satisfy
 the stated assumptions, then
\begin{equation}\label{E:ofi1}
\limsup_{N\to\infty}\sup_{\norm{f_0}_\infty, \norm{f_2}_{\infty} ,\ldots,\norm{f_{\ell}}_\infty\leq 1}
\E_{1\leq n\leq N} \Big|\int f_0\cdot  T^{[a_1(n)]}f_1\cdot T^{[a_2(n)]}f_2\cdot\ldots \cdot T^{[a_\ell(n)]}f_{\ell} \ d\mu\Big| \ll_{\ell,a_1,\ldots,a_\ell} \nnorm{f_1}_{2\ell}.
\end{equation}
We also claim an analogous estimate with $f_i$ in place of $f_1$ for $i=0,2,\ldots,\ell$.
We leave it to the reader to verify  that such estimates follow  from \eqref{E:ofi1}
(for $i=2,\ldots,\ell$ by symmetry, for $i=0$ we factor out the transformation $T^{[a_1(n)]}$
and work with the resulting averages, the precise argument is very similar to the one given in
the  beginning of the proof of Lemma~\ref{L:uniform}).

First we verify that  \eqref{E:ofi1}  implies \eqref{E:paok1}. We can assume that $\norm{f_i}_\infty\leq 1$ for $i=1,\ldots,\ell$. Since
$f_1$ is orthogonal to $\cZ$, we have  $\nnorm{f_1}_{2\ell}=0$, and as a
consequence the limsup in \eqref{E:ofi1} is $0$. Using \eqref{E:ofi1} with
the conjugate of the function $\E_{1\leq n\leq
N}T^{[a_1(n)]}f_1\cdot T^{[a_2(n)]}f_2\cdot\ldots \cdot
T^{[a_\ell(n)]}f_{\ell}$ in place of the function $f_0$ (for every $N\in \N$),
and removing the norms we get \eqref{E:paok1}.

We proceed now to prove \eqref{E:ofi1} by induction. Suppose first
that  $\ell=1$. If $a_1(t)\prec t$, then we  deduce from
Lemma~\ref{L:ChangeVar}  that
\begin{equation}\label{E:ofi1a}
\limsup_{N\to\infty} \sup_{\norm{f_0}_\infty\leq 1} \E_{1\leq n\leq
N} \Big|\int f_0\cdot  T^{[a_1(n)]}f_1 \ d\mu\Big|\ll_{a_1}
\limsup_{N\to\infty}\sup_{\norm{f_0}_\infty\leq 1} \E_{1\leq n\leq
N} \Big|\int f_0\cdot T^{n}f_1 \ d\mu\Big| .
\end{equation}
Equation \eqref{E:ofi1} now follows by combining \eqref{E:ofi1a} and the  estimate in Lemma~\ref{L:ccx}.
If $a_1(t)\sim t$, then $a_1(t)=\alpha t+e(t)$ for some nonzero $\alpha\in \R$ and $e(t)\prec t$.
Assuming that  $\alpha>0$ (the other case can be treated similarly),
one  sees that
limit in \eqref{E:ofi1} is bounded by a constant (one can use
$([\alpha]+1)$ if $\alpha>1$ and $2$ if $\alpha\leq 1$) times the
quantity
$$
\limsup_{N\to\infty}\sup_{\norm{f_0}_\infty\leq 1}
\E_{1\leq n\leq N} \Big|\int f_0\cdot  T^nf_1 \ d\mu\Big|.
$$
Combining this with Lemma~\ref{L:ccx} gives the advertised estimate.

Suppose now that $\ell\geq 2$ and the statement holds for $\ell-1$.
 We first claim that when proving \eqref{E:ofi1} we can  assume that the function $a_1(t)$
 has maximal growth.  Indeed, if  $a_1(t)\prec a_i(t)$ for some $i\in 2,\ldots,\ell$,
then we can factor out the transformation $T^{[a_i(n)]}$ and work with the resulting average.
We omit the details since the argument is very similar to the one given in the beginning of the proof of
Lemma~\ref{L:uniform}.

 We consider two cases:

{\bf Case 1.} Suppose  that $a_1(t)\sim t$.
We can assume that  for some $r\in \{1,\ldots,\ell\}$ we have
$a_1(t)=\alpha_1t +b_1(t),\ldots,  a_r(t)=\alpha_rt+ b_r(t)$ where $\alpha_i$ are non-zero real numbers,
$b_i(t)\prec t$,
and $a_{r+1}(t),\ldots, a_\ell(t)\prec t$.

We choose functions $f_{i,N}$ with $\norm{f_{i,N}}_\infty\leq 1$  for $i=0,2,\ldots,\ell$, such that the value of the averages in \eqref{E:ofi1}  is $1/N$ close to the supremum. Using the Cauchy-Schwarz inequality we see
that \eqref{E:ofi1} follows if we show that
\begin{equation}\label{E:ofi2}
\limsup_{N\to\infty}
\E_{1\leq n\leq N}\Big| \int f_{0,N}\cdot  T^{[a_1(n)]}f_1\cdot T^{[a_2(n)]}f_{2,N}\ldots \cdot T^{[a_\ell(n)]}f_{\ell,N} \ d\mu\Big|^2 \ll_{\ell,a_1,\ldots,a_\ell} \nnorm{f_1}^2_{2\ell}.
\end{equation}
 Notice that the limit in
 \eqref{E:ofi2} is equal to
 \begin{equation}\label{E:ofi3}
A=\limsup_{N\to\infty}
\E_{1\leq n\leq N} \int F_{0,N}\cdot  S^{[a_1(n)]}F_1\cdot S^{[a_2(n)]}F_{2,N}\cdot \ldots \cdot S^{[a_\ell(n)]}F_{\ell,N} \ d(\mu\times \mu)
\end{equation}
where
$$
S=T\times T,\quad   F_1=f_1 \otimes \overline{f}_1, \ \text{ and } \ \  F_{i,N}=f_{i,N} \otimes \overline{f}_{i,N}
\ \text{ for } \ i=0,2,\ldots,\ell.
$$
Using first the Cauchy Schwarz inequality, and then Lemma~\ref{L:N-VDC}, we see  that
\begin{equation}\label{E:ofi3.5}
|A|^2\leq 4\ \! \limsup_{H\to\infty}
\E_{1\leq h\leq H}A_h
\end{equation}
where
\begin{multline*}
A_h=\limsup_{N\to\infty}
\E_{1\leq n\leq N} \Big|\int   S^{[a_1(n+h)]}F_1\cdot S^{[a_2(n+h)]}F_{2,N}\cdot \ldots \cdot S^{[a_\ell(n+h)]}F_{\ell,N} \cdot\\
S^{[a_1(n)]}\overline{F}_1\cdot S^{[a_2(n)]}F_{2,N}\cdot \ldots \cdot S^{[a_\ell(n)]}\overline{F}_{\ell,N}
\ d(\mu\times \mu)\Big|.
\end{multline*}
 We factor out the transformation  $S^{[a_1(n)]}$. For $h\in \N$ fixed and  large $n\in \N$,  using that
 $b_i(n+h)-b_i(n)\to 0$ for $i=1,\ldots,r$ (since $b_i(t)\prec t$),
and  $a_i(n+h)-a_i(n)\to 0$ for $i=r+1,\ldots,\ell$ (since $a_i(t)\prec t$),
we get the identities
\begin{gather*}
  [a_i(n+h)]-[a_1(n)]=
 [a_i(n)-a_1(n)]+[\alpha_i h]+e_i(h,n) \ \text{ for } i=1,\ldots,r,\\
 [a_i(n+h)]-[a_1(n)]=[a_i(n)-a_1(n)]+e_i(h,n) \ \text{ for } i=r+1,\ldots, \ell, \\
[a_i(n)]-[a_1(n)]=[a_i(n)-a_1(n)]+\tilde{e}_i(h,n) \ \text{ for } i=2,\ldots,\ell,
\end{gather*}
where the error terms $e_i(h,n)$ and  $\tilde{e}_i(h,n)$
take values in the set $\{0,\pm 1, \pm 2\}$.

We find that
\begin{multline*}
A_h=\limsup_{N\to\infty}
\E_{1\leq n\leq N} \Big|\int   \tilde{F}_{1,h,n}\cdot
S^{[a_2(n)-a_1(n)]}
\tilde{F}_{2,h,n,N}\cdot \ldots
\cdot S^{[a_\ell(n)-a_1(n)]} \tilde{F}_{\ell,h,n,N}
\ d(\mu\times \mu)\Big|
\end{multline*}
where $\tilde{F}_{1,h,n}= T^{[\alpha_1 h]+e_1(h,n)}F_1\cdot  \overline{F_1}$, $\tilde{F}_{i,h,n,N}= T^{[\alpha_i h]+e_i(h,n)}F_{i,N}\cdot  T^{\tilde{e}_i(h,n)}\overline{F_{i,N}}$
for $i=2, \ldots, r$, and
$\tilde{F}_{i,h,n,N}= T^{e_i(h,n)}F_{i,N}\cdot  T^{\tilde{e}_i(h,n)}\overline{F_{i,N}}$
for $i=r+1, \ldots, \ell$.
  Next notice that for every fixed $h\in \N$ we can partition the integers into a finite number of sets,
  that depend only on $\ell$,  where all sequences $e_i(h,n), \tilde{e}_i(h,n)$ are constant.
  It follows that for every $h\in\N$ there exists $i\in\{0,\pm 1, \pm 2\}$ such that
\begin{multline*}
A_h\ll_\ell
\limsup_{N\to \infty}
\sup_{\norm{F_2}_{\infty} ,\ldots,\norm{F_{\ell}}_\infty\leq 1}\E_{1\leq n\leq N} \Big|\int   \tilde{F}_{1,h,i}\cdot
S^{[a_2(n)-a_1(n)]} F_2\cdot \ldots
\cdot S^{[a_\ell(n)-a_1(n)]} F_\ell
\ d(\mu\times \mu)\Big|
\end{multline*}
where $\tilde{F}_{1,h,i}=T^{[\alpha_1 h]+i}F_1\cdot  \overline{F_1}$.
Since the functions $a_2(t)-a_1(t),\ldots, a_\ell(t)-a_1(t)$ satisfy the assumptions of the induction hypothesis, it follows  that
\begin{equation}\label{E:ofi4}
A_h\ll_{\ell,a_2-a_1,\ldots,a_\ell-a_1}  \max_{i=0,\pm 1,\pm 2}\ \nnorm{\tilde{F}_{1,h,i}}_{2(\ell-1)}.
\end{equation}
($\nnorm{\tilde{F}_{1,h,i}}_{k}$ is constructed using the system $(X\times X,\X\times \X, S, \mu\times\mu)$.)
Since
$$
\tilde{F}_{1,h,i}=T^{[\alpha_1 h]+i}F_1\cdot  \overline{F_1}
=T^{[\alpha_1 h]+i}(f_1\otimes \overline{f_1})\cdot
(\overline{f_1}\otimes f_1)=
 (T^{[\alpha_1 h]+i}f_1\cdot \overline{f_1}) \otimes
 (\overline{T^{[\alpha_1 h]+i}f_1\cdot \overline{f_1}}),$$
 and \eqref{E:seminonergodic} gives
$$
\nnorm{f\otimes \overline{f}}_{k}\leq \nnorm{f}^2_{k+1}
$$
for every $f\in L^\infty(\mu)$ and $k\in\N$,  we conclude that
\begin{equation}\label{E:ofi5}
\nnorm{\tilde{F}_{1,h,i}}_{2(\ell-1)}\leq \nnorm{T^{[\alpha_1
h]+i}f_1\cdot \overline{f_1}}^2_{2\ell-1}.
\end{equation}
Combining \eqref{E:ofi4} and  \eqref{E:ofi5} we get
\begin{equation}\label{E:ofi6}
\limsup_{H\to\infty}
\E_{1\leq h\leq H} A_h\ll_{\ell,a_2-a_1,\ldots,a_\ell-a_1} \max_{i=1,..,5} \limsup_{H\to\infty}
\E_{1\leq h\leq H} \nnorm{T^{[\alpha_1 h]+i}f_1\cdot \overline{f_1}}^2_{2\ell-1}.
\end{equation}
Finally, one  sees that
\begin{equation}\label{E:ofi7}
 \limsup_{H\to\infty}
\E_{1\leq h\leq H} \nnorm{T^{[\alpha_1 h]+i}f_1\cdot \overline{f_1}}^2_{2\ell-1}
\leq ([\alpha_1]+1) \limsup_{H\to\infty}
\E_{1\leq h\leq H} \nnorm{T^{h}f_1\cdot \overline{f_1}}^2_{2\ell-1},
\end{equation}
and using H\"{o}lder's inequality we get
\begin{equation}\label{E:ofi8}
\limsup_{H\to\infty}
\E_{1\leq h\leq H} \nnorm{T^{h}f_1\cdot \overline{f_1}}^2_{2\ell-1}\leq \limsup_{H\to\infty}
\Big(\E_{1\leq h\leq H} \nnorm{T^{h}f_1\cdot \overline{f_1}}_{2\ell-1}^{2^{2\ell-1}}\big)^{\frac{1}{2^{2\ell-2}}}=\nnorm{f_1}_{2\ell}^4.
\end{equation}
(The last equality follows from \eqref{eq:recur}.)
Combining \eqref{E:ofi3.5}, \eqref{E:ofi6}, \eqref{E:ofi7}, and \eqref{E:ofi8},  we deduce that
$$
|A|\ll_{\ell,a_1,\ldots,a_{\ell}}\nnorm{f_1}_{2\ell}^2.
$$
This proves \eqref{E:ofi2} and
completes the induction step in the case where $a_1(t)\sim t$.

{\bf Case 2.} It remains to deal with the case  $a_1(t)\prec t$.
For $i=1,\ldots,\ell$, we write $a_i(t)=\tilde{a}_i(a_1(t))$ where $\tilde{a}_i(t)=a_i(a_1^{-1}(t))$. Notice that
 $\tilde{a}_i(t)$ is an element of  $\mathcal{P}$ since $\mathcal{P}$ is closed under composition and compositional inversion. Keeping in mind that $a_1(t)$ has maximal growth and using Lemma~2.7 and Proposition~2.11 in  \cite{BH09}, we get for some $\varepsilon>0$  that $t^\varepsilon\prec \tilde{a}_i(t)\ll t$ for $i=1,\ldots,\ell$
 and  $t^\varepsilon\prec \tilde{a}_i(t)-\tilde{a}_j(t)\ll t$ for $i\neq j$. Hence, the functions $\tilde{a}_1(t)=t,\tilde{a}_2(t),\ldots,\tilde{a}_\ell(t)$ satisfy the assumptions of the induction hypothesis.

 Using  Lemma~2.7 in \cite{BH09}  we get for $i=2,\ldots,\ell$ that
$$
[a_i(n)]=[\tilde{a}_i([a_1(n)])]+e_i(n)
$$
where the error terms $e_i(n)$ take values on a finite set. Therefore, the  left hand side in \eqref{E:ofi1} is equal to
$$
\limsup_{N\to\infty}\sup_{\norm{f_0}_\infty, \norm{f_2}_{\infty} ,\ldots,\norm{f_{\ell}}_\infty\leq 1}
\E_{1\leq n\leq N} \Big|\int f_0\cdot  T^{[a_1(n)]}f_1\cdot  T^{[\tilde{a}_2([a_1(n)])]+e_2(n)}f_2\cdot \ldots \cdot T^{[\tilde{a}_\ell([a_1(n)])]+e_\ell(n)}f_{\ell} \ d\mu\Big|.
$$
Since the error terms $e_i(n)$ take values on a finite set, say with cardinality  $K$
(with $K$ depending on the $a_i$'s only), one sees  that
 the previous limit is bounded by $K^{\ell-1}$  times
\begin{equation}\label{E:lmz}
\limsup_{N\to\infty}\sup_{\norm{f_0}_\infty, \norm{f_2}_{\infty} ,\ldots,\norm{f_{\ell}}_\infty\leq 1}
\E_{1\leq n\leq N} \Big|\int f_0\cdot  T^{[a_1(n)]}f_1\cdot  T^{[\tilde{a}_2([a_1(n)])]}f_2\cdot \ldots \cdot T^{[\tilde{a}_\ell([a_1(n)])]}f_{\ell} \ d\mu\Big|.
\end{equation}
Since $t^\varepsilon\prec a_1(t)\prec t$, we can apply
Lemma~\ref{L:ChangeVar} for the sequence of real numbers
$(V_N(n))_{N,n\in\N}$ defined by
$$
V_N(n)=\Big|\int f_{0,N}\cdot  T^{n}f_{1}\cdot
T^{[\tilde{a}_2(n)]}f_{2,N}\cdot \ldots \cdot
T^{[\tilde{a}_\ell(n)]}f_{\ell,N} \ d\mu\Big|,
$$
where the functions $f_{i,N}$ are chosen so that the average in \eqref{E:lmz} gets $1/N$ close to the supremum.
 We get that the   limit in \eqref{E:lmz} is bounded by a constant
 (that depends only on $a_1$) times
$$
\limsup_{N\to\infty}\sup_{\norm{f_0}_\infty, \norm{f_2}_{\infty} ,\ldots,\norm{f_{\ell}}_\infty\leq 1}
\E_{1\leq n\leq N}  \Big|\int f_0\cdot  T^nf_1\cdot  T^{[\tilde{a}_2(n)]}f_2\cdot \ldots \cdot T^{[\tilde{a}_\ell(n)]}f_{\ell} \ d\mu\Big|.
$$
Since the functions $t,\tilde{a}_2(t),\ldots,\tilde{a}_\ell(t)$
satisfy the assumptions of the induction hypothesis, we are reduced to
Case $1$. This completes  the induction step and the proof.
\end{proof}

\subsection{The general case}
We are going to use an inductive method analogous to the one used in Section~\ref{SS:PET1} for polynomial families. First, following \cite{BH09}, we introduce a notion of complexity that is suitable for our current setup.

We remind the reader that $\mathcal{G}=\{a\in C(\R_+)\colon t^{k+\varepsilon}\prec a(t)\prec t^{k+1} \text{ for some } k\geq 0 \text{ and } \varepsilon>0\}$.
We say that a family $\mathcal{F}=\{a_1(t),\ldots a_\ell(t)\}$ of functions in
$\mathcal{LE}$ is \emph{``nice''} if  $a_i(t)\in \mathcal{G}$ and $a_i-a_j\in \mathcal{G}$ for $i\neq j$.
  Given any such ``nice'' family
  we associate a vector $(d,n_d,\ldots,n_1,n_0)$, called the \emph{type of $\mathcal{F}$}, with
non-negative integer entries, as follows: For every non-negative integer $i$
 let
 $$
 \mathcal{F}_i=\{a\in \mathcal{F}\colon t^i\prec  a(t)\prec t^{i+1}\}.
 $$
 We say that two functions $a,b\in \mathcal{F}_i$ are equivalent if $a(t)-b(t)\prec t^i$, and
  we define $n_i$ to be the number of non-equivalent elements of $\mathcal{F}_i$. We denote  the maximum integer $i$ for which $n_i\neq 0$ by
 $d$. We order the set of all possible types
lexicographically, meaning,
$(d,w_d,\ldots,w_1)>(d',w_d',\ldots,w_1')$ if and only if in the
first instance where the two vectors disagree the coordinate of the
first vector is greater than the coordinate of the second vector.

\begin{example} If $\mathcal{F}=\{t^{1/3}, t^{5/2}, t^{5/2}+t^{1/2}, t^{5/2}+t^{7/3}\}$,
 then the  second and third functions are equivalent and all the other functions are non-equivalent.
 Hence, the type of $\mathcal{F}$ is $(2,2,0,1)$.
\end{example}
Given a ``nice'' family of functions  $\mathcal{F}=\{a_1(t),\ldots a_\ell(t)\}$, a positive integer
$h\in \N$, and $a\in \mathcal{F}$, we form a new family $\mathcal{F}(a(t),h)$ as follows:
We
start with the family of polynomials
$$
\{ a_1(t+h)-a(t),\ldots,a_\ell(t+h)-a(t), a_1(t)-a(t),\ldots, a_\ell(t)-a(t)\},
$$
and successively remove the smallest number of functions so that the remaining set consists of unbounded
functions whose pairwise differences are also unbounded. Then for every large $h$
the family $\mathcal{F}(a,h)$ is also ``nice'' (if non-empty). The
 function
$a_i(t+h)-a(t)$ will be removed if and only if   $a_i(t)\prec t$, and the function $a_i(t)-a(t)$ will be removed if and only if $a(t)=a_i(t)$.
\begin{example}
If  $\mathcal{F}=\{t^{1/3},t^{1/2},  t^{3/2}\}$ and $a(t)=t^{1/3}$, then
we
start with the family of polynomials $$\{(t+h)^{1/3} -t^{1/3}, (t+h)^{1/2}- t^{1/3}, (t+h)^{3/2}- t^{1/3}, t^{1/3}- t^{1/3}, t^{1/2} -t^{1/3}, t^{3/2} -t^{1/3}\}
$$
and remove the first, second, and fourth functions to get
$$
\mathcal{F}(t^{1/3},h)=\{(t+h)^{3/2}- t^{1/3}, t^{1/2} -t^{1/3}, t^{3/2} -t^{1/3}\}.
$$
Notice that the family $\mathcal{F}$ has type $(1,1,2)$, and the family
$\mathcal{F}(t^{1/3},h)$ has smaller type, namely, $(1,1,1)$.
\end{example}

To prove Theorem~\ref{T:CharSeveral} we are going to use induction on the type of
the family of functions involved.
In order to carry out the inductive step we will use the following:
\begin{lemma}\label{L:indstep}
Let $\mathcal{F}=\{a_1(t),\ldots a_\ell(t)\}$ be a ``nice'' family of functions.
 Suppose that  $a_1(t)\succ t$, and $a_1(t)$ has maximal growth rate within $\mathcal{F}$, meaning
 $a_i(t)\ll a_1(t)$ for $i=2,\ldots, \ell$.

Then it is possible to choose
$a\in \mathcal{F}$ such that for every large $h$ the family $\mathcal{F}(a(t),h)$
is ``nice'',  has type smaller than that of $\mathcal{F}$, and the function $a_1(t+h)-a(t)$ has maximal growth rate within the family $\mathcal{F}(a(t),h)$.
\end{lemma}
\begin{remark} Since $a_1(t)\succ t$,  no-matter what the choice of the function $a(t)$ will be, the function
$a_1(t+h)-a(t)$ is going to be an element of the family $\mathcal{F}(a(t),h)$ for every large $h$.
\end{remark}
\begin{proof}
Suppose first that $a_i(t)\prec a_1(t)$ for some $i\in\{2,\ldots,\ell\}$. Let $i_0$ be such that the function $a_{i_0}(t)$ has the minimal growth (meaning $a_{i_0}(t)\ll a_i(t)$ for $i=1,\ldots,\ell$). Then $a(t)=a_{i_0}(t)$ has the advertised property.

Otherwise,  $a_i(t)\sim a_1(t)$ for $i=1,\ldots,\ell$, in which case, for $i=2,\ldots,\ell$, we have  $a_i(t)=\alpha_ia_1(t)+b_i(t)$, for some
  non-zero real numbers $\alpha_2,\ldots, \alpha_\ell$, and functions $b_i(t)$ with $b_i(t)\prec a_1(t)$.
If $\alpha_{i_0}\neq 1$ for some $i_0\in \{2,\ldots,\ell\}$, then  $a(t)=a_{i_0}(t)$ has the advertised property. If $\alpha_i= 1$ for $i=2,\ldots,\ell$, let
 $i_0$ be such that the function $b_{i_0}(t)$ has maximal growth.  Then $a(t)=a_{i_0}(t)$ has the advertised property. This completes the proof.
\end{proof}

We are now ready to give the proof of Theorem~\ref{T:CharSeveral}.
We recall its statement for convenience.
\begin{theorem}
Suppose that $\{a_1(t),\ldots,a_\ell(t)\}$  is  a  ``nice'' family of functions in $\LE$.

 Then the factor $\cZ$ is characteristic for the family
 $\{[a_1(n)], \ldots, [a_\ell(n)]\}$.
\end{theorem}
\begin{proof}
Arguing as in the beginning of the proof of Proposition~\ref{L:uniform}
we see that it suffices to show the following:  If
$(X,\X,\mu , T)$ is a measure preserving system,  $f_1\in
L^\infty(\mu)$ is orthogonal to the nilfactor $\mathcal{Z}$, the
family of functions $\mathcal{F}=\{a_1(t),\ldots,a_\ell(t)\}$ is nice, and the function
$a_1(t)$ has maximal growth within $\mathcal{F}$,  then
\begin{equation}\label{E:pao0}
\lim_{N\to\infty}\sup_{\norm{f_0}_\infty, \norm{f_2}_{\infty}
,\ldots,\norm{f_\ell}_\infty\leq 1} \E_{1\leq n\leq N} \Big|\int
f_0\cdot  T^{[a_1(n)]}f_1\cdot T^{[a_2(n)]}f_2\cdot\ldots \cdot
T^{[a_\ell(n)]}f_\ell \ d\mu\Big| =0.
\end{equation}
We are going to use induction on  the type of the family
$\{a_1(t),\ldots,a_\ell(t)\}$ to verify  this statement.

 Proposition~\ref{P:CharSublin} shows that the
result holds for $d=0$.
 Suppose now that $d\geq 1$, and our statement  holds for all families with
  type smaller than $(d,n_d,\ldots,n_0)$. Let $\mathcal{F}=\{a_1(t),\ldots,a_\ell(t)\}$ be
  a family with type $(d,n_d,\ldots,n_0)$ (then  $a_1(t)\succ t$ since
  $a_1(t)$ has maximal growth rate in $\mathcal{F}$).

Choosing functions $f_{i,N}$ with $\norm{f_{i,N}}_\infty \leq 1$, so that the average in \eqref{E:pao0} is $1/N$ close to the supremum,
and using the Cauchy-Schwarz  inequality, we see that  it suffices to show that
$$
\lim_{N\to\infty}
\E_{1\leq n\leq N}\Big| \int f_{0,N}\cdot  T^{[a_1(n)]}f_1\cdot T^{[a_2(n)]}f_{2,N}\ldots \cdot T^{[a_\ell(n)]}f_{\ell,N} \ d\mu\Big|^2=0.
$$
 Equivalently, it suffices to show that
 $$
\lim_{N\to\infty}
\E_{1\leq n\leq N} \int F_{0,N}\cdot  S^{[a_1(n)]}F_1\cdot S^{[a_2(n)]}F_{2,N}\cdot \ldots \cdot S^{[a_\ell(n)]}F_{\ell,N} \ d(\mu\times \mu)=0
$$
where $S=T\times T$, $F_1=f_1 \otimes \overline{f}_1$, and $F_{i,N}=f_{i,N} \otimes \overline{f}_{i,N}$, for $i=0,2,\ldots,\ell$.
Using the Cauchy-Schwarz inequality, and then Lemma~\ref{L:N-VDC}, we reduce matters to showing
for every large  $h$ that
\begin{multline*}
\lim_{N\to\infty}
\Big| \E_{1\leq n\leq N} \int   S^{[a_1(n+h)]}F_1\cdot S^{[a_2(n+h)]}F_{2,N}\cdot \ldots \cdot S^{[a_\ell(n+h)]}F_{\ell,N} \cdot\\
S^{[a_1(n)]}\overline{F}_1\cdot S^{[a_2(n)]}\overline{F}_{2,N}\cdot \ldots \cdot S^{[a_\ell(n)]}\overline{F}_{\ell,N}
\ d(\mu\times \mu)\Big|=0.
\end{multline*}
Factoring out the transformation $S^{[a(n)]}$
 where $a(t)\in \{a_1(t),\ldots,a_\ell(t)\}$ is as in Lemma~\ref{L:indstep}, we see that it suffices to show that   for
 every large $h$  we have
  \begin{multline*}
\lim_{N\to\infty}
\E_{1\leq n\leq N} \Big|\int   S^{[a_1(n+h)-a(n)]+e_1(n)}F_1\cdot S^{[a_2(n+h)-a(n)]+e_2(n)}F_{2,N}\cdot \ldots \cdot S^{[a_\ell(n+h)-a(n)]+e_\ell(n)}F_{\ell,N} \cdot\\
S^{[a_1(n)-a(n)]+e_{\ell+1}(n)}\overline{F}_1\cdot S^{[a_2(n)-a(n)]+e_{\ell+2}(n)}\overline{F}_{2,N}\cdot \ldots \cdot S^{[a_\ell(n)-a(n)]+e_{2\ell}(n)}\overline{F}_{\ell,N}
\ d(\mu\times \mu)\Big|=0
\end{multline*}
for some error terms $e_i(n)$ that take values  in $\{0,1\}$.
Therefore, it suffices to show that for every large $h$ we have
  \begin{multline}\label{E:pao2}
\lim_{N\to\infty}\sup_{\norm{F_2}_{\infty}
,\ldots,\norm{F_{2\ell}}_\infty\leq 1}
\E_{1\leq n\leq N} \Big|\int   S^{[a_1(n+h)-a(n)]}F_1\cdot S^{[a_2(n+h)-a(n)]}F_2\cdot \ldots \cdot S^{[a_\ell(n+h)-a(n)]}F_\ell \cdot\\
S^{[a_1(n)-a(n)]}F_{l+1}\cdot S^{[a_2(n)-a(n)]}F_{l+2}\cdot \ldots \cdot S^{[a_\ell(n)-a(n)]}F_{2\ell}
\ d(\mu\times \mu)\Big|=0.
\end{multline}
Next notice that if $a_i(t)\prec t$ for some $i\in \{2,\ldots,\ell\}$, then $a_i(n+h)-a_i(n)\to 0$. Therefore,  for our purposes we can practically assume that $[a_i(n+h)-a(n)]=[a_i(n)-a(n)]$ for all $n$, and  for those values of $i$
we can write
$$
S^{[a_i(n+h)-a(n)]}F_i\cdot S^{[a_i(n)-a(n)]}\overline{F}_i=S^{[a_i(n)-a(n)]}|F_i|^2.
$$
After doing this reduction, for every large $h$
 the  iterates that appear in the integral in \eqref{E:pao2}  involve functions that belong to the family
 $\mathcal{F}(a(t),h)$ (defined at the beginning of the current subsection). This family  is ``nice'' and by Lemma~\ref{L:indstep}
  has type smaller than that of $\mathcal{F}$. Furthermore,
 one of these iterates appearing in this reduced form is  $S^{[a_1(n+h)-a(n)]}F_1$  and the function $a_1(n+h)-a(n)$ has maximal growth in $\mathcal{F}(a(t),h)$ . Since $f_1$ is orthogonal to the factor $\cZ(T)$, we have that  $F_1=f_1\otimes \overline{f}_1$ is orthogonal to the factor $\cZ(S)$. Therefore, the induction hypothesis applies, and gives that the limit
 in \eqref{E:pao2} is $0$ for every large $h$. This completes the induction and the proof.
\end{proof}

\section{Proof of the convergence and  the recurrence results}\label{S:ProofsMain}
 In this section we combine Theorems~\ref{T:CharSingle} and \ref{T:CharSeveral}  and the equidistribution results from \cite{Fr09} to prove the convergence and recurrence results stated in Sections~\ref{SS:ArithmeticProgressions} and \ref{SS:SeveralSequences}.

\subsection{Convergence  results}
We prove the convergence results stated in Sections~\ref{SS:ArithmeticProgressions} and \ref{SS:SeveralSequences}.
To prove  Theorem~\ref{T:ConvSingle} we will need the following result:
\begin{theorem}[{\bf F.~\cite{Fr09}}]\label{T:A}
Suppose that the function $a\in \H$ has polynomial growth and satisfies one of the three conditions
stated in  Theorem~\ref{T:ConvSingle}.

 Then for every  nilmanifold $X=G/\Gamma$,
$F\in C(X)$, $b\in G$, and $x\in X$, the   following limit exists
$
\lim_{N\to\infty}\frac{1}{N}\sum_{n=1}^N F(b^{[a(n)]}x)
$.
\end{theorem}

\begin{proof}[Proof of Theorem~\ref{T:ConvSingle}]
The necessity of the conditions was proved in \cite{BKQW05}
by using examples of rational rotations on the circle.

To show that the  three stated conditions are sufficient for convergence we start by using  Theorem~\ref{T:CharSingle}. We get that the nilfactor $\cZ$
 is characteristic for the corresponding multiple ergodic averages. Using an ergodic decomposition argument  and  Theorem~\ref{T:HoKra},
 we deduce that it suffices to prove convergence when our system is an inverse limit of nilsystems. A standard approximation argument allows us to finally reduce matters to nilsystems.

 Let $(X=G/\Gamma,\mathcal{G}/\Gamma,m_X,T_b)$ be a nilsystem and $F_1,\ldots, F_\ell\in L^\infty(m_X)$.
Our goal is to show that if the function $a\in \H$ satisfies one of the three stated conditions, then the limit
 \begin{equation}\label{E:reducednil1}
\lim_{N\to\infty} \E_{1\leq n\leq N} F_1(b^{[a(n)]}x)\cdot F_2(b^{2[a(n)]}x)\cdot \ldots \cdot F_\ell(b^{\ell[a(n)]}x)
 \end{equation}
 exists in $L^2(m_X)$. By density, we can assume that the functions $F_1,\ldots, F_\ell$ are continuous. In this case we claim  that the limit in \eqref{E:reducednil1} exists for every $x\in X$. Indeed, applying Theorem~\ref{T:A}
 to the nilmanifold $X^k$,  the nilrotation $\tilde{b}=(b,b^2,\ldots,b^\ell)$, the point $\tilde{x}=(x,x,\ldots,x)$, and  the function $\tilde{F}(x_1,\ldots,x_\ell)=F_1(x_1)\cdot F_2(x_2)\cdots F_\ell(x_l)$, we get that
 the limit
 $$
 \lim_{N\to\infty} \E_{1\leq n\leq N}\tilde{F}(\tilde{b}^{[a(n)]}\tilde{x})
 $$
 exists. This implies that the limit in  \eqref{E:reducednil1} exists for every $x\in X$ and completes the proof.
\end{proof}
To prove  Theorem~\ref{T:ConvSingleFormula} we will need the following result:
\begin{theorem}[{\bf F.~\cite{Fr09}}]\label{T:B}
Let $a\in\H$ have  at most polynomial growth and satisfy   $|a(t)-cp(t)|\succ \log t$ for every $c\in \R$ and  $p\in \Z[t]$.

Then for every nilmanifold   $X=G/\Gamma$,   $b\in G$, and   $x\in X$, the
  sequence $(b^{[a(n)]}x)_{n\in\N}$ is equidistributed in the nilmanifold $ \overline{(b^nx)}_{n\in \N}$.
 \end{theorem}

\begin{proof}[Proof of Theorem~\ref{T:ConvSingleFormula}]
Arguing  as in the proof of Theorem~\ref{T:ConvSingle} we reduce matters to showing the following:
Let $a\in \H$ satisfy the assumptions of our theorem,  $X=G/\Gamma$ be a nilsystem, $b\in G$ be a nilrotation,  and $F_1,\ldots, F_\ell\in C(X)$.  Then for every $x\in X$ the limit in \eqref{E:reducednil1}
exists and is equal to the limit
$$
\lim_{N\to\infty} \E_{1\leq n\leq N} F_1(b^nx)\cdot F_2(b^{2n}x)\cdot \ldots \cdot F_\ell(b^{\ell n}x).
$$
Keeping the same notation as in the proof of Theorem~\ref{T:ConvSingleFormula}, we  rewrite the desired identity as
$$
\lim_{N\to\infty} \E_{1\leq n\leq N}\tilde{F}(\tilde{b}^{[a(n)]}\tilde{x})=
\lim_{N\to\infty} \E_{1\leq n\leq N}\tilde{F}(\tilde{b}^{n}\tilde{x}).
$$
Using Theorem~\ref{T:B}, and the fact that
 the sequence $(\tilde{b}^n \tilde{x})_{n\in\N}$ is equidistributed in the set
$\overline{\{\tilde{b}^n \tilde{x}, n\in\N\}}$, we deduce that the last identity holds for every $\tilde{x}\in X^\ell$, completing the proof.
\end{proof}
Next we prove Theorem~\ref{T:ConvSeveral} using the following result that will be established in Section~\ref{SS:proof}.
\begin{proposition}\label{P:L^2product}
Suppose that the functions $a_1,\ldots,a_\ell \in\mathcal{LE}\cap \G$ have different growth rates.

Then for every nilmanifold $X=G/\Gamma$, $b\in G$ ergodic, and $F_1,\ldots,F_\ell \in C(X)$ we have
$$
\lim_{N\to\infty}\E_{1\leq n\leq N} F_1(b^{[a_1(n)]}x)\cdot \ldots \cdot F_\ell(b^{[a_\ell(n)]}x)
=\int F_1 \ d{m_X}\cdot\ldots\cdot \int F_\ell \ d{m_X}
$$
where the convergence takes place in $L^2(m_X)$.
\end{proposition}

\begin{proof}[Proof of Theorem~\ref{T:ConvSeveral}]
Using Theorem~\ref{T:CharSeveral} and arguing  as in the proof of Theorem~\ref{T:ConvSingle} the result follows from
 Proposition~\ref{P:L^2product}.
\end{proof}

Lastly, we prove Theorem~\ref{T:ConvNonCommuting}. Its proof has a rather soft touch of  ergodic theory; in fact
it is an easy consequence of the following general statement:
\begin{proposition}\label{P:ConvSequence}
Let $a_1,\ldots,a_\ell\in \LE\cap \mathcal{G}$  have different growth rates and satisfy $a_i(t)\prec t$ for $i=1,\ldots,\ell$.
Let $(X,\mathcal{X},\mu)$  be a probability space, and for $i=1,\ldots,\ell$ let  $(F_i(n))_{n\in\N}$ be  sequences of
functions in $L^\infty(\mu)$ with uniformly bounded norm
such that the limits
 $\tilde{F}_i=\lim_{N-M\to\infty}\E_{M\leq n\leq N}F_i(n)$ exist in $L^2(\mu)$.

 Then
 $$
\lim_{N\to\infty} \E_{1\leq n\leq N} F_1([a_1(n)])\cdot\ldots \cdot F_\ell([a_\ell(n)])=
\tilde{F}_1\cdot \ldots \cdot \tilde{F}_\ell
 $$
where  the limit is taken in $L^2(\mu)$.
 \end{proposition}
\begin{proof}
As in the proof of Proposition~\ref{P:CharSublin} it will be convenient  to work with the
larger Hardy field $\mathcal{P}$ of Pfaffian functions which  is closed under composition and compositional
inversion.

We use induction on $\ell$. For $\ell=1$ the result follows from
Lemma~\ref{L:ChangeVar}.

Suppose that the result holds for $\ell-1$. Without loss of generality we can assume that
$a_i(t)\prec a_\ell(t)$ for $i=1,\ldots,\ell-1$. Assuming that $\tilde{F}_\ell=0$ it suffices to prove
that
\begin{equation}\label{E:zcx}
\lim_{N\to\infty}   \E_{1\leq n\leq N} F_1([a_1(n)])\cdot  \ldots
  \cdot F_\ell( [a_\ell(n)])=0
\end{equation}
where the convergence takes place in $L^2(\mu)$. For
$i=1,\ldots\ell-1$ we let $\tilde{a}_i(t)=a_i(a_\ell^{-1}(t))$ which
is again an element of $\mathcal{P}$. Since the functions
$a_1(t),\ldots, a_\ell(t)$ have different growth rates, belong to
$\mathcal{G}$, and $a_\ell(t)$ has maximal growth, the same is the
case for the functions $
\tilde{a}_1(t),\ldots,\tilde{a}_{\ell-1}(t),t$, and
$\tilde{a}_i(t)\prec t$ for $i=1,\ldots,\ell-1$. Keeping  this in
mind and using  Lemma~2.7 in \cite{BH09} we get that
$t^\varepsilon\prec\tilde{a}_i(t)\prec t$  for some $\varepsilon>0$.
Furthermore, by Lemma 2.12 in \cite{BH09} we have
$[a_i(n)]=\tilde{a}_i([a_\ell(n)])$ for a set of $n\in \N$ of
density $1$. It follows from  Lemma~\ref{L:ChangeVar} that in order
to verify \eqref{E:zcx} it suffices to show that
\begin{equation}\label{E:zcx'}
\lim_{N\to\infty}   \E_{1\leq n\leq N} F_1([\tilde{a}_1(n)])\cdot  \ldots\cdot
  F_{\ell-1}([\tilde{a}_{\ell-1}(n)]) \cdot F_\ell(n)=0.
\end{equation}
Since for $i=1,\ldots,\ell-1$ the functions $\tilde{a}_i(t)$ belong to some Hardy field  and  satisfy $\tilde{a}_i(t)\prec t$ we have  that $\tilde{a}_i(t+1)-\tilde{a}_i(t)\to 0$. Using this, we see that
 there exists a sequence $(I_m)_{m\in\N}$ of   non-overlapping intervals, with $|I_m|\to \infty$, $\bigcup_{m=1}^\infty I_m=\N$, and
such  that the sequences $[\tilde{a}_1(n)],\ldots,[\tilde{a}_{\ell-1}(n)]$ are constant on every interval $I_m$ (for technical reasons we can also assume that $|I_m|\prec m$).
Then  \eqref{E:zcx'} follows if we show that
\begin{equation}\label{E:zcx''}
\lim_{M\to\infty}  \E_{1\leq m \leq M}\big( G_m\cdot  \E_{n\in I_m}  F_\ell(n)\big)=0,
\end{equation}
where
$$
G_m
= F_1([\tilde{a}_1(n_m)])\cdot  \ldots \cdot
  F_{\ell-1}([\tilde{a}_{\ell-1}(n_m)]),
$$
and $n_m$ is any element of the interval $I_m$.
Furthermore, since  the functions $G_m$ have uniformly bounded $L^\infty$ norms,
 it suffices to show that
\begin{equation}\label{E:zcx'''}
\lim_{M\to\infty}  \E_{1\leq m \leq M}\norm{\E_{n\in I_m}  F_\ell(n)}_{L^2(\mu)}=0.
\end{equation}
Our assumption gives that
$$
\lim_{m\to\infty} \norm{\E_{n\in I_m}  F_\ell(n)}_{L^2(\mu)}=0,
$$
which  immediately implies \eqref{E:zcx'''}. This  completes the proof.
\end{proof}
We deduce Theorem~\ref{T:ConvNonCommuting} from Proposition~\ref{P:ConvSequence}.
 \begin{proof}[Proof of Theorem~\ref{T:ConvNonCommuting}]
  By the (uniform) mean ergodic theorem we have
$$
\lim_{N-M\to\infty} \norm{\E_{M\leq n\leq N}  T_i^nf_i}_{L^2(\mu)}=E(f_i|\mathcal{I}(T_i))
$$
where the convergence takes place in $L^2(\mu)$.
We can therefore apply Proposition~\ref{P:ConvSequence} for the sequences $(F_i(n))_{n\in\N}$, $i=1,\ldots,\ell$,  defined by $F_i(n)=T^n_if_i$ to conclude the proof.
 \end{proof}

\subsection{Recurrence results}
We prove Theorem~\ref{T:RecSingle} using the following  multiple recurrence result that will be handled later.
\begin{proposition}\label{P:closetopoly}
Let $a\in \H$ be of the form $a(t)=p(t)\alpha+e(t)$ for some $p\in\Z[t]$, $\alpha\in \R$, and $1\prec e(t)\prec t$.
Let $(X,\X,\mu,T)$ be a system and $f\in L^\infty(\mu)$ be non-negative and not almost everywhere zero.

Then for every $\ell \in \N$ we have
\begin{equation}\label{E:1821}
\limsup_{N-M\to \infty} \E_{M\leq  n \leq  N} \int f\cdot T^{[a(n)]}f\cdot \ldots\cdot T^{\ell[a(n)]}f \ d\mu>0.
\end{equation}
\end{proposition}
\begin{proof}[Proof of Theorem~\ref{T:RecSingle}]
If $|a(t)-cp(t)|\succ \log{t}$ for every $c\in \R$ and $p\in\Z[t]$, then the result follows immediately
by combining Theorem~\ref{T:ConvSingleFormula} and Furstenberg's multiple recurrence theorem (\cite{Fu77}).
Furthermore, the case where $a(t)=cp(t)+e(t)$ for some $c\in \R$, $p\in\Z[t]$, and $e(t)\ll \log{t}$,
is taken care by  Proposition~\ref{P:closetopoly}. This completes the proof.
\end{proof}
Lastly, we deduce  Theorem~\ref{T:RecMult} from Theorem~\ref{T:ConvSeveral}.
\begin{proof}[Proof of Theorem~\ref{T:RecMult}]
Let $\mu=\int \mu_t d\lambda(t)$ be the ergodic decomposition of  the measure $\mu$. Using
 Theorem~\ref{T:ConvSeveral} for the ergodic systems $(X,\X,\mu_t,T)$ we get that
$$
\lim_{N\to\infty} \E_{1\leq n\leq N}\mu(A\cap T^{-[a_1(n)]}A\cap T^{-[a_2(n)]}A\cap\cdots \cap
T^{-[a_\ell(n)]}A)=\int (\mu_t(A))^{\ell +1} \ d\lambda(t).
$$
Using Holder's inequality we see that the last integral is at least
$$
\Big(\int \mu_t(A)\ d\lambda(t)\Big)^{\ell+1}=(\mu(A))^{\ell +1}.
$$
This proves the advertised estimate and completes the proof.
\end{proof}

\subsection{Proof of Proposition~\ref{P:L^2product}}\label{SS:proof} In the case  where all  the functions
 $a_1(t),\ldots,a_\ell(t)$ have super-linear growth  Proposition~\ref{P:L^2product} is a direct consequence
 of the corresponding pointwise result in \cite{Fr09} (Theorem~1.3).
The general case can be covered using a modification of an argument
used in \cite{Fr09}. To avoid unnecessary repetition we only sketch
the proof.

\begin{proposition}\label{P:Rn+r}
Suppose that the functions  $a_1,\ldots,a_\ell \in\mathcal{LE}$ have
different growth rates and satisfy $t^{k_i}\log{t} \prec a_i(t)\prec
t^{k_i+1}$ for some $k_i\in \N$.

Then for every nilmanifold $X=G/\Gamma$,   $b\in G$,   $x\in X$, and
$F\in C(X^\ell)$, we have
$$
\lim_{R\to\infty}\limsup_{N\to\infty}\E_{1\leq n\leq N}\Big|
\E_{1\leq r\leq
R}F(b^{[a_1(Rn+r)]}x,\ldots,b^{[a_\ell(Rn+r)]}x)-\int F \
dm_{X^\ell_b}\Big|=0
$$
where   $X_b=\overline{\{b^nx\colon n\in\N\}}$.
\end{proposition}
\begin{proof}[Sketch of Proof]
Using a straightforward modification of the reduction argument  of
Section 5.2 in \cite{Fr09}, we can reduce matters to showing that
for every nilmanifold $X=G/\Gamma$,  with $G$ connected and simply
connected,  $b\in G$ ergodic,   and $F\in C(X^\ell)$,
 we have
$$
\lim_{R\to\infty}\limsup_{N\to\infty}\E_{1\leq n\leq N}\Big|
\E_{1\leq r\leq
R}F(b^{[a_1(Rn+r)]}x,\ldots,b^{[a_\ell(Rn+r)]}x)-\int F \
dm_{X^\ell}\Big|=0,
$$
This was verified while proving Proposition~5.3  in \cite{Fr09},
completing the proof.
\end{proof}

\begin{proof}[Proof of Proposition~\ref{P:L^2product}]

Our goal is to show that for every nilmanifold $X=G/\Gamma$,  $b\in
G$ ergodic, and $F_1,\ldots,F_\ell \in C(X)$, we have
$$
\lim_{N\to\infty}\E_{1\leq n\leq N} F_1(b^{[a_1(n)]}x)\cdot \ldots
\cdot F_\ell(b^{[a_\ell(n)]}x) =\int F_1 \ d{m_X}\cdot\ldots\cdot
\int F_\ell \ d{m_X}
$$
where the convergence  takes place in $L^2(m_X)$.

If all the functions $a_1(t),\ldots,a_\ell(t)$ are sub-linear, then   the result follows from  Theorem~7.3 in \cite{BH09}, and if all the functions are super-linear, then  the result follows from Theorem~1.3 in \cite{Fr09}.
If none of these is the case,  we can assume that  $a_1(t),\ldots,a_m(t)\succ t$ and $a_{m+1}(t),\ldots, a_\ell(t)\prec t$ for some $m\in \{1,\ldots,\ell-1\}$.

Let $F\in C(X^m)$ and $G\in C(X^{\ell-m})$.  It suffices to show
that if $\int F\ \!dm_{X^m}=0$, 
then
\begin{equation}\label{EE:main}
\lim_{N\to\infty}\E_{1\leq n\leq N} \big(F
({b}^{[a_1(n)]}x,\ldots,{b}^{[a_{m}(n)]}x)\cdot G
({b}^{[a_{m+1}(n)]}x,\ldots,{b}^{[a_{\ell}(n)]}x)\big)=0
\end{equation}
where convergence takes place in $L^2(m_X)$. For every $R\in\N$ the
 averages ijn \eqref{EE:main} are asymptotically equal to the averages
\begin{equation}\label{EE:main1}
\E_{1\leq n\leq N} \Big(\E_{1\leq r\leq R} \big(
F(b^{[a_1(nR+r)]}x,\ldots,b^{[a_m(nR+r)]}x)\cdot
G(b^{[a_{m+1}(nR+r)]}x,\ldots,b^{[a_{\ell}(nR+r)]}x)\big)\Big).
\end{equation}
Since  the functions $a_{m+1},\ldots, a_\ell\in \LE$
are all sub-linear,
we can  show (see Lemma 2.12 in \cite{BH09})  the following:  for every  $R\in \N$,  for a set of $n\in \N$ of density $1$,
 we have $[a_i(nR+r)]=[a_i(nR)]$ for $r=1,\ldots,R$  and $i=m+1,\ldots,\ell$.  We deduce that  the limsup as $N\to\infty$  of the
 $L^2(m_X)$-norm of the averages  in \eqref{EE:main1}
 is bounded by a constant times
\begin{equation}\label{EE:main2}
\limsup_{N\to\infty} \E_{1\leq n\leq N} \norm{
 \E_{1\leq r\leq R}
F(b^{[a_1(nR+r)]}x,\ldots,b^{[a_m(nR+r)]}x)}_{L^2(m_X)}.
\end{equation}
Using Proposition~\ref{P:Rn+r} we see that the limit of the
expression \eqref{EE:main2}  as  $R\to\infty$ is $0$, completing the
proof.
\end{proof}
\subsection{Proof of Proposition~\ref{P:closetopoly}}
First, we  informally discuss the proof strategy of Proposition~\ref{P:closetopoly}.
When the function $a(t)$ is logarithmically close to a constant multiple of an integer polynomial  the relevant multiple ergodic averages cannot be directly compared with  Furstenberg's
 averages (because of the luck of equidistribution).  To bypass this difficulty we  work with an appropriate subsequence of the sequence $[a(n)]$. This subsequence is chosen so that
 it becomes possible to compare the corresponding multiple ergodic averages with
 those along the sequence $[n \alpha ]$. For the latter averages
positiveness follows easily from Furstenberg's multiple recurrence theorem, thus achieving our goal.

To carry out this plan we need  a few  preliminary results that enable us to carry out the aforementioned ``comparison step''. 

We start with an equidistribution result on nilmanifolds. To prove the claimed uniformity we make use some quantitative equidistribution results  (this is the only place in the present article where we make explicit use of such results).
\begin{lemma}\label{L:unif}
Let $X=G/\Gamma$ be a connected nilmanifold, $b\in G$ be an ergodic nilrotation, and $p\in \Z[t]$ be  non-constant.

Then for every  $F\in C(X)$   we have
 \begin{equation}\label{E:unif}
\lim_{N-M\to\infty}\max_{x\in X}\Big| \E_{M\leq n\leq N } F(b^{p(n)}x)-\int F \ dm_X\Big|=0.
\end{equation}
\end{lemma}
\begin{proof}
We argue by contradiction. Suppose that  \eqref{E:unif} fails for some connected nilmanifold $X$, ergodic $b\in G$, $p\in \Z[t]$ with  $k=\deg(p)\geq 1$, and $F\in C(X)$. Then there exist
  $\delta>0$, $x_m\in X$, and sequences of positive integers $(n_m)_{m\in\N}$, $(N_m)_{m\in\N}$,
 with $N_m\to \infty$, such that the sequence
$(b^{p(n_m+n)}x_m)_{1\leq n\leq N_m}$ is not $\delta$-equidistributed in $X$ for every $m\in\N$.

Then by Theorem~\ref{T:GT2} (suppose that $\delta$ is small enough so that the theorem applies) there exists  a constant $M=M(X,\delta,k)$, and a sequence of  quasi-characters $\psi_m$,
with $\norm{\psi_m}\leq M$, and  such that
\begin{equation}\label{E:M}
\norm{\psi_m(b^{p(n_m+n)}x_m)}_{C^\infty[N_m]}\leq M
\end{equation}
for every $m\in\N$.
As explained in Section~\ref{SS:BackgroundNil}, the affine torus $A$ of $X$ can
  be identified with a finite dimensional torus $\T^l$. After making
  this identification, we have $\psi_m(t)=\kappa_m\cdot t$ for some
  non-zero $\kappa_m\in \Z^l$, and the nilrotation $b$ induces a $d$-step
  unipotent affine transformation $T_b\colon \T^l\to \T^l$.
  Let $\pi(b)=(\beta_1\Z,\ldots,\beta_s\Z)$, where $\beta_i\in \R$,  be the projection of $b$ to the Kronecker factor of
  $T_b$ (notice that
$s$ is bounded by the dimension of $X$).
Since $b$ acts ergodically on $X$ the set  $B=\{1, \beta_1,\ldots,\beta_s\}$ is rationally independent.
For every $x\in\T^l$ the coordinates of $T_b^n x$
are polynomials of $n$, and so $\kappa_m\cdot T_b^nx$ is a
polynomial of $n$. Moreover, the leading
term of the polynomial $\kappa_m\cdot T_b^nx$ has the form $ \gamma_m n^{\tilde{d}}$,
where $1\leq \tilde{d}\leq d$, and
\begin{equation}\label{E:beta}
  \gamma_m=\frac{1}{k!}\sum_{i=1}^s r_{m,i}\beta_i, \ \text{ where } r_i\in \Z \text{ are not all  zero and } |r_{m,i}|\leq c_1\cdot M
\end{equation}
for some constant $c_1$ that depends only on $b$.\footnote{For example, if $T\colon \T^2\to \T^2$ is defined by
$T(x,y)=(x+\alpha,y+2x+\alpha)$ and $k=(k_1,k_2)$, then $k\cdot T^n(x,y)=k_2\alpha n^2+(k_1\alpha+2k_2x)n$. Therefore,
the corresponding leading term is either $k_2\alpha n^2$ (if $k_2\neq 0$), or $k_1\alpha n$ (if $k_2=0$).} Using this and the
definition of $\norm{\cdot}_{C^\infty[N]}$ (see \eqref{E:norms}), it
follows that
$$
\norm{\psi(T_b^{p(n_m+n)}x_m)}_{C^\infty[N_m]}\geq
N_m^{k\tilde{d}} \norm{\gamma_m}.
$$
Combining this with \eqref{E:M} we get that
\begin{equation}\label{E:m}
  \norm{\gamma_m}\leq \frac{M}{N_m^{k\tilde{d}}}.
\end{equation}
Since  by \eqref{E:beta} we have  a finite number of
options for the irrational numbers $\gamma_m$, and $N_m\to \infty$,
letting $m\to \infty$ we  get a contradiction. This completes the proof.
\end{proof}

Next we use the  Lemma~\ref{L:unif} to establish a key identity. We remark that if
 $G$ is a connected and simply connected nilpotent Lie group, then
  there exists  a unique continuous homomorphism $b\colon \R\to G$ with $b(1)=b$. For $b\in G$ and  $t\in \R$,
 by $b^t$ we mean the element $b(t)\in G$.
\begin{lemma}\label{L:key}
Let $X=G/\Gamma$ be a nilmanifold with $G$ connected and simply connected,  $\alpha\in \R$ be  non-zero,
and  $b\in G$ be such that
the sub-nilmanifold $\overline{\{(n\alpha\Z ,b^{n\alpha}\Gamma),n\in\N\}}$
   of $\T\times X$ is connected.

Then for every $F\in C(X)$, F{\o}lner sequence $(\Phi_m)_{m\in\N}$ in $\Z$, and  non-constant $p\in\Z[t]$  we have
\begin{equation}\label{E:key}
\lim_{m\to\infty}\E_{n\in\Phi_m} F\big(b^{[p(n) \alpha+m\alpha]}\Gamma \big)=
\lim_{N\to \infty}\E_{1\leq n\leq N} F(b^{[n \alpha ]}\Gamma).
\end{equation}
\end{lemma}
\begin{proof}
We first do some maneuvers that enable us to remove the integer parts and bring us to a point where Lemma~\ref{L:unif} is applicable. We let $\tilde{X}$ be the nilmanifold  $\T\times X$ which we identify with $(\R\times G)/(\Z\times \Gamma)$,
set $\tilde{b}=(\alpha,b^\alpha)$, and define the  function $\tilde{F}$ on $\tilde{X}$  by
$$
\tilde{F}(t\Z,g\Gamma)=F(b^{-\{t\}}g\Gamma).
$$
 Notice that for every $t\in \R$ we have
\begin{equation}\label{E:bvc}
\tilde{F}(\tilde{b}^{t}\tilde{\Gamma})=F(b^{-\{\alpha t\}}b^{\alpha t}\Gamma)=F(b^{[\alpha t]}\Gamma)
\end{equation}
where $\tilde{\Gamma}=\Z\times \Gamma$.
(We caution the reader that the function $\tilde{F}$ is not continuous on $\tilde{X}$.)

By assumption the nilmanifold $\tilde{X}=\overline{\{\tilde{b}^{n}\Gamma, n\in\N\}}$
is connected.
 Since the nilrotation $\tilde{b}$ acts ergodically on the connected  nilmanifold $\tilde{X}$, part of the hypothesis of Lemma~\ref{L:unif} is satisfied.

 Next we claim that Lemma~\ref{L:unif} can be applied for  the restriction of the function $\tilde{F}$ to the nilmanifold $\tilde{X}$, namely we claim that
 \begin{equation}\label{E:ada}
\lim_{N-M\to\infty}\max_{\tilde{x}\in \tilde{X}}\Big| \E_{M\leq  n\leq N } \tilde{F}(\tilde{b}^{p(n)}\tilde{x})-\int \tilde{F} \ dm_{\tilde{X}}\Big|=0.
\end{equation}
Suppose first that   $\alpha$ is rational.
Since $\tilde{X}$ is a connected nilmanifold and $\tilde{\Gamma}\in \tilde{X}$, we have    that  $\tilde{X}\subset \{\Z\}\times X$.
  Hence, the restriction of  $\tilde{F}$ onto $\tilde{X}$ is given by $\tilde{F}(\Z,g\Gamma)=F(g\Gamma)$ (for those
  $(\Z,g\Gamma)$  that belong to $\tilde{X}$), and as a result is  continuous. In this case,  \eqref{E:ada} is a direct consequence of Lemma~\ref{L:unif}.
   Therefore,  it remains to verify \eqref{E:ada}  when $\alpha$ is irrational.
   Notice first that the set of discontinuities of $\tilde{F}$ on $\tilde{X}$ is a subset of the nilmanifold
   $\{\Z\}\times X$. Near
a point $(\Z,g\Gamma)$ of  $\{\Z \}\times X$ the function  $\tilde{F}$ comes close to  the value
$F(g\Gamma)$ or the value $F(b^{-1}g\Gamma)$.
For $\delta>0$ (and smaller than $1/2$) let $\tilde{X}_\delta=I_\delta \times X$ where
$I_\delta=\big\{t\Z\colon \norm{t}\geq \delta \big\}$.
There exist functions  $\tilde{F}_\delta\in C(\tilde{X})$ that agree with $\tilde{F}$ on $\tilde{X}_\delta$ and have sup-norm bounded by $2\norm{F}_\infty$. Since $\alpha$ is irrational,   the sequence $(p(n)\alpha\Z)_{n\in\N}$
(which happens to be the first coordinate of $(\tilde{b}^{p(n)})_{n\in\N}$) is well distributed  in $\T$, and as a result
 $$
 \lim_{N-M\to\infty}\max_{x\in[0,1]}|\E_{M\leq  n\leq N} {\bf 1}_{[\delta,1-\delta]}(\{x+p(n)\alpha\}) -(1-2\delta)|=0.
 $$
It follows that
\begin{equation}\label{E:approx}
\limsup_{N-M\to\infty} \max_{\tilde{x}\in \tilde{X}}\E_{M\leq  n\leq N}|\tilde{F}(\tilde{b}^{p(n)}\tilde{x})- \tilde{F}_\delta(\tilde{b}^{p(n)}\tilde{x})|\leq 3\norm{F}_\infty \delta.
\end{equation}
  By Lemma~\ref{L:unif}, equation \eqref{E:ada}
holds if one uses the functions $\tilde{F}_\delta$ in place of the function $\tilde{F}$. Using this and
\eqref{E:approx}, we deduce \eqref{E:ada}.
 This proves our claim.

Next notice that \eqref{E:ada} gives that
\begin{equation}\label{E:form1}
\lim_{m\to\infty}\E_{n\in\Phi_m} \tilde{F}(\tilde{b}^{p(n)+m}\tilde{\Gamma})=\int
\tilde{F} \ dm_{\tilde{X}}.
\end{equation}
Furthermore, since the nilrotation $\tilde{b}$ acts ergodically in $\tilde{X}$,
the last integral is  equal to
\begin{equation}\label{E:form2}
\lim_{N\to\infty}\E_{1\leq  n\leq N}\tilde{F}(\tilde{b}^n\tilde{\Gamma}).
\end{equation}
(If $\tilde{F}$ is not continuous we argue as before to get this.)
Using  \eqref{E:bvc}
   we see that
\begin{equation}\label{E:form3}
\tilde{F}(\tilde{b}^{p(n)+m}\tilde{\Gamma})=F\big(b^{[p(n)\alpha+m\alpha]}\Gamma \big), \quad \text{ and } \quad \tilde{F}(\tilde{b}^n\tilde{\Gamma})=F(b^{[n\alpha]}\Gamma).
\end{equation}
Combining \eqref{E:form1}, \eqref{E:form2}, and \eqref{E:form3}, we get \eqref{E:key}.
This completes the proof.
\end{proof}

The next lemma will be used later   to verify that certain connectedness assumptions (needed to apply Lemma~\ref{L:key})
 are satisfied.
\begin{lemma}\label{L:Ziegler}
Let $X=G/\Gamma$ be a connected nilmanifold and $b\in G$ be an ergodic nilrotation.


Then there exists a connected sub-nilmanifold $Z$ of $X^\ell$ such that for a.e. $g\in G$ the element $\tilde{b}_g=(g^{-1}bg,g^{-1}b^2g,\ldots,g^{-1}b^\ell g)$ acts ergodically on $Z$.
\end{lemma}
\begin{remark}
The independence of $Z$ of the generic $g\in G$ will not be used, only that $Z$ is connected will be used.
\end{remark}
 \begin{proof}
This is an immediate consequence of a limit formula that appears  in Theorem~2.2 of \cite{Zi05}
(the details of the deduction appear in Corollary~2.10 of \cite{Fr08}).
\end{proof}

As mentioned before, our plan is to  establish  a positiveness property by comparing certain multiple ergodic averages to  some simpler ones involving the sequence $[n\alpha]$. The next lemma establishes the positiveness property needed for the latter averages.
\begin{lemma}\label{L:[na]}
Let $(X,\X, \mu ,T)$ be a system, $\ell\in \N$,  and $\alpha$ be a non-zero real number.

Then for every $f\in L^\infty(\mu)$ positive and not almost everywhere zero we have
 \begin{equation}\label{E:skl}
\liminf_{N\to\infty} \E_{1\leq n\leq N } \int f \cdot T^{[n \alpha ]}f\cdot \ldots\cdot T^{\ell [n \alpha ]}f \ d\mu>0.
\end{equation}
\end{lemma}
\begin{proof}
We follow closely an argument used in \cite{BH96} (Theorem 2.3).

For a general sequence  $(a(n))_{n\in\N}$ of non-negative numbers, and $m\in\N$,  one has
$$
\liminf_{N\to\infty} \E_{1\leq n\leq N} a(n)
\geq \frac{1}{m} \liminf_{N\to\infty}\E_{1\leq n\leq N} a(mn).
$$
Using this for $m=1,\ldots,N_0$, where $N_0$ is an integer that will be chosen later,  and averaging over $m$ we get
$$
\liminf_{N\to\infty} \E_{1\leq n\leq N} a(n)
\geq  \E_{1\leq m\leq N_0}\Big(\frac{1}{m} \liminf_{N\to\infty}\E_{1\leq n\leq N} a(mn)\Big).
$$
 The last expression is greater or equal than
$$
\frac{1}{N_0}\liminf_{N\to\infty}\E_{1\leq n\leq N} ( \E_{1\leq m\leq N_0}  a(mn)).
$$
As a result, if there exists $N_0\in \N$ such that
$$
\E_{1\leq m\leq N_0}  a(mn)>c>0
$$
for a set $S$ of  $n\in\N$ of positive lower density, then
$$
\liminf_{N\to\infty} \E_{1\leq n\leq N} a(n)>0.
$$
We shall use this for
 $$
a(n)= \int f \cdot T^{[n \alpha ]}f\cdot \ldots\cdot T^{\ell [n\alpha ]}f \ d\mu.
$$
to prove \eqref{E:skl}.

We choose $N_0$ as follows: By the uniform multiple recurrence property (\cite{BHRF00}),  there exists  $N_0\in \N$, and positive constant
$c$ (depending only on $\int f d\mu$ and $\ell$),  such that
for every $r\in \N$ one has
\begin{equation}\label{E:uni}
\E_{1\leq m\leq N_0 } \int f \cdot T^{rm}f\cdot \ldots\cdot T^{\ell r m}f \ d\mu\geq c.
\end{equation}

We choose $S=\big\{n\colon \{n\alpha\}< 1/N_0\big\}$, which as is well known has positive density.
Since
$[m\beta]=m[\beta]$ whenever $\{\beta\}<1/m$, we have
$$
a(mn)=\int f \cdot T^{m[n \alpha]}f\cdot \ldots\cdot T^{\ell m[n \alpha]}f \ d\mu
$$
for every $n\in S$ and $m=1,\ldots,N_0$. As a result, for every $n\in S$ we have
$$
\E_{1\leq m\leq N_0}  a(mn)=\E_{1\leq m\leq N_0}
\int f \cdot T^{m[n \alpha]}f\cdot \ldots\cdot T^{\ell m[n \alpha ]}f \ d\mu\geq  c
$$
where the last estimate follows form  \eqref{E:uni} applied  for  $r=[n \alpha ]$. As explained
 before this implies \eqref{E:skl} and  completes the proof.
\end{proof}
We now have all the ingredients needed to prove Proposition~\ref{P:closetopoly}.
\begin{proof}[Proof of Proposition~\ref{P:closetopoly}]
If $\alpha=0$, or $p$ is constant, we have $1\prec a(t)\prec t$. Since $a\in \H$, it  follows  that
for every large $m$ the sets $\{n\in\N\colon [a(n)]=m\}$ are intervals with length
increasing to infinity as $m\to\infty$. Using this,  the result follows  from Furstenberg's multiple recurrence theorem (\cite{Fu77}).

Suppose now that $\alpha\neq 0$ and $p$ is non-constant.
We start with some reductions. We can assume that $p(0)=0$. Furthermore, using an ergodic decomposition argument we can assume that the system is ergodic. By Theorem~\ref{T:CharSingle} we can reduce matters to showing \eqref{E:1821} in the case where the system is an inverse limit of nilsystems. Lastly, an argument completely analogous to the one used
in the proof of  Lemma~3.2 in
\cite{FuK79} shows that the positiveness property \eqref{E:1821} is preserved by inverse limits. Hence, we can  assume that the system is a nilsystem.

  Suppose now that $X=G/\Gamma$ is a nilmanifold,  $b\in G$, and $F\in L^\infty(m_X)$  is non-negative and not almost everywhere zero.
   Our goal is to show that there exists a  sequence of intervals $(I_m)_{m\in\N}$ with length converging to infinity such that
  \begin{equation}\label{E:wanted}
\lim_{m\to\infty}  \E_{n\in I_m} V([a(n)])> 0
\end{equation}
where
\begin{equation}\label{E:V(n)}
 V(n)=\int F(x) \cdot F(b^nx)\cdot \ldots \cdot F(b^{\ell n}x)    \ dm_X.
  \end{equation}

The choice of the intervals  $(I_m)_{m\in\N}$
 will be made so that the values of the sequence  $[a(n)]$ for  $n\in I_m$ take a convenient form.
More precisely,
since $1\prec e(t)\prec t$ and $e\in \H$,  for every $r\in \N$ (to be chosen later), there exists a sequence  of intervals $(J_{r,m})_{m\in\N}$ with $|J_{r,m}|\to\infty$,  such that
$$
mr\alpha\leq e(n) \text{  for } n\in J_{r,m} \ \text{ and } \ \sup_{n\in J_{r,m}}( e(n) -mr\alpha)\to 0.
$$
For $\delta>0$ (to be chosen later)  define
$$
S_{r,\delta}=\{n\in \N\colon \{p(n)r\alpha\}\leq  1-\delta\}.
$$
Then for  every large $m$, if   $n\in J_{r,m}\cap S_{r,\delta}$ we have
$$
[a(n)] =[p(n)\alpha+mr\alpha].
$$
 It follows that  in order to verify \eqref{E:wanted} it suffices to show that for some $r\in \N$ and $\delta>0$ we have
\begin{equation}\label{E:577}
\limsup_{m\to\infty}\E_{n\in J_{r,m}} (V([p(n)\alpha+mr\alpha])\cdot {\bf 1}_{S_{r,\delta}}(n))>0.
\end{equation}
In fact it suffices to show that there exists $r\in \N$ such that
\begin{equation}\label{E:577'}
\limsup_{m\to\infty}\E_{n\in J_{r,m}} V([p(n)\alpha+mr\alpha])>0,
\end{equation}
 the reason being that for every $r\in \N$, $\alpha\in \R$, and $\delta$ small, we have
  $$
\limsup_{m\to \infty} \E_{n\in J_{r,m}} (1- {\bf 1}_{S_{r,\delta}}(n))\leq \delta,\footnote{If $\alpha$ is irrational
the estimate holds with equality (this follows from equidistribution). If $\alpha$ is rational the estimate holds trivially for $\delta$ small.}
$$
and as a result  \eqref{E:577'} readily implies \eqref{E:577}, again if  $\delta$ is small.
Therefore, we can concentrate our attention in proving \eqref{E:577'}.

We shall  show that there exists an $r\in \N$
such that
\begin{equation}\label{E:key'}
\lim_{m\to\infty}\E_{ n\in \Phi_m} V\big([p(rn)\alpha+mr\alpha] \big)=
\lim_{N\to \infty}\E_{1\leq n\leq N} V([nr\alpha])
\end{equation}
for every F{\o}lner sequence $(\Phi_m)_{m\in\N}$ in $\Z$.
Having shown this, we can finish the proof  by using  the definition of $V(n)$ and Lemma~\ref{L:[na]}.
We deduce  that the limits in \eqref{E:key'}  are positive, and as a result \eqref{E:577'} holds.

We proceed now to find this $r\in \N$.     Using a standard lifting argument we can also assume that the group $G$ is connected and simply connected.\footnote{We can assume that $G/G_0$ is finitely generated. In this case, one can show (\cite{Lei05a})  that there exists
 a nilmanifold $\hat{X}=\hat{G}/\hat{\Gamma}$, where
 $\hat{G}$ is a connected and simply-connected Lie group, such that
 for every  $F\in C(X)$, $b\in G$, and $x\in X$,   there exists
 $\hat{F}\in C(\hat{X})$, $\hat{b}\in \hat{G}$, and $\hat{x}\in \hat{X}$, such that
 $F(b^nx)=\hat{F}(\hat{b}^n\hat{x})$ for every $n\in\N$.}
A standard approximation argument shows that  it suffices to  verify \eqref{E:key'} whenever the function  $F$ in  \eqref{E:V(n)} is continuous.
  Our plan is to use Lemma~\ref{L:key} to establish a stronger pointwise result, namely, there exists an $r\in \N$ such that for a.e. $x\in X$ we have for every F{\o}lner sequence $(\Phi_m)_{m\in\N}$ in $\Z$ that
  \begin{equation} \label{E:pointwise}
\lim_{m\to\infty}\E_{n\in \Phi_m} F(b^{[p(rn)\alpha+mr\alpha]}x)
    \cdot\ldots\cdot F(b^{\ell [p(rn)\alpha+mr\alpha]}x)=
\lim_{m\to\infty}\E_{1\leq n\leq N} F(b^{[nr\alpha]}x)
    \cdot\ldots\cdot F(b^{\ell [nr\alpha]}x).
    \end{equation}
  (Our argument will show that both limits exist.) Trivially, this pointwise
  result  implies \eqref{E:key'}. Therefore,  we are left with establishing \eqref{E:pointwise}.

We do some preparation in order to use Lemma~\ref{L:key}.
For every  $n\in\N$ and element $x=g\Gamma$ of $ X$ we have
\begin{equation} \label{E:r_0A}
F(b^{n}x)
    \cdot\ldots\cdot F(b^{\ell n}x)=
\tilde{F}(\tilde{b}^{n}_g\tilde{\Gamma}),
    \end{equation}
  where $\tilde{F}(x_1,\ldots,x_l)=F(gx_1)\cdot \ldots\cdot F(gx_\ell)$ ($\in C(X^\ell)$),
   $\tilde{\Gamma}=\Gamma^\ell$, and
   $$
  \tilde{b}_g=(g^{-1}bg,g^{-1}b^{2 }g,\ldots, g^{-1}b^{\ell }g) \in G^\ell.
   $$
Next we show that there exists an $r\in\N$ such that for a.e. $g\in G$ the sub-nilmanifold
  \begin{equation}\label{E:connected}
  \overline{\{(nr\alpha,\tilde{b}^{nr\alpha}_g\tilde{\Gamma}), n\in\N\}}
  \end{equation} of $\T\times X^\ell$ is connected.
Indeed,  as mentioned in Section~\ref{SS:BackgroundNil},  there exists  $r_0$ such that the nilmanifold
  $\overline{\{(nr_0\alpha\Z,b^{nr_0\alpha}\Gamma),n\in\N\}}$ is connected.
It follows form  Lemma~\ref{L:Ziegler} (applied to the nilmanifold $\T\times X$ and
 the element $(r_0\alpha,b^{r_0\alpha})\in \T\times G$) that   for almost every $g\in G$ the   sub-nilmanifold
$$
  \overline{\{(nr_0\alpha \Z,\ldots,\ell nr_0\alpha \Z,\tilde{b}^{nr_0\alpha}_g\tilde{\Gamma}), n\in\N\}}
$$
of $\T^\ell \times X^\ell$   is connected.  Projecting in the appropriate coordinates we get  that the nilmanifold in \eqref{E:connected} is also connected.

We are now in a position  where we can apply Lemma~\ref{L:key} for the nilmanifold $X^\ell$, the function $\tilde{F}\in C(X^\ell)$,  the integer $r_0$ we just found, the element $r_0\alpha\in \R$ in place of $\alpha$, the polynomial $q\in\Z[t]$ defined by $q(t)=p(r_0t)/r_0$ (remember $p(0)=0$),  and  the elements
$\tilde{b}_g\in G^\ell$ (for those $g\in G$ for which the set \eqref{E:connected} is connected for $r=r_0$). We deduce  that for almost every $g\in G$, and  for every F{\o}lner sequence $(\Phi_m)_{m\in\N}$ in $\Z$, we have that
$$
 \lim_{m\to\infty}\E_{n\in \Phi_m} \tilde{F}({\tilde{b}_g}^{[p(r_0n)\alpha+mr_0\alpha]}\tilde{\Gamma})
=\lim_{N\to\infty}\E_{1\leq n\leq N} \tilde{F}({\tilde{b}_g}^{[nr_0\alpha ]}\tilde{\Gamma}).
$$
Using \eqref{E:r_0A}, this gives \eqref{E:pointwise}, which in turn gives \eqref{E:key'}.
 This completes the proof.
\end{proof}

\end{document}